\documentclass[reqno,a4paper]{amsart}

\usepackage{amsmath,amssymb}

\usepackage[utf8]{inputenc}
\usepackage[pdftex,linktocpage,pdfstartview=FitH,colorlinks=true,allcolors=blue]{hyperref}
\usepackage{amsthm} 
\usepackage[capitalise]{cleveref}
\usepackage[stretch=10, shrink=10]{microtype}
\usepackage{commath} 
\usepackage[initials,msc-links]{amsrefs}[2007/10/22]

\usepackage{tikz}
\usetikzlibrary{decorations.markings}
\usetikzlibrary{snakes}
\usepackage[labelfont=rm]{subcaption}
\usepackage{paralist,enumitem}
\setlist{listparindent=0pt,parsep=3pt}
\setdefaultenum{1.}{}{}{}

\renewcommand{\d}{\mathop{}\!\mathrm{d}}
\newcommand{\Ortho}{\operatorname{O}}
\newcommand{\Aut}{\operatorname{Aut}}
\newcommand{\dual}[1]{\check{#1}}
\DeclareMathOperator{\Ad}{Ad} 
\newcommand{\R}{\mathbb{R}}

\newcommand{\Z}{\mathbb{Z}}
\newcommand{\N}{\mathbb{N}}
\renewcommand{\P}{\mathbb{P}}
\newcommand{\A}{\mathbb{A}}
\DeclareSymbolFont{script}{U}{eus}{m}{n}
\DeclareSymbolFontAlphabet{\mathscr}{script}
\DeclareMathSymbol{\EuWedge}{0}{script}{"5E}
\newcommand{\Wedge}{\EuWedge}
\newcommand{\fo}{\mathfrak{o}}
\newcommand{\tens}{\otimes}
\newcommand{\cL}{\mathcal{L}}
\newcommand{\restr}[1]{{}_{|#1}}
\newcommand{\st}{\mathrel{|}}
\newcommand{\half}{\tfrac12}
\newcommand{\QQ}{\mathcal{Q}}
\renewcommand{\q}{\mathfrak{q}}
\newcommand{\p}{\mathfrak{p}}
\renewcommand{\t}{\mathfrak{t}}
\newcommand{\spn}[1]{\langle#1\rangle}
\newcommand{\ijkl}{\ell kji}
\newcommand{\prl}{\mathrel{\Vert}}
\newcommand{\cZ}{\mathcal{Z}}
\renewcommand{\c}{\mathfrak{c}}
\newcommand{\abrack}[1]{[\mkern-3mu[#1]\mkern-3mu]}

\newcommand{\dom}{\Sigma}

\newcommand{\K}{Koenigs}
\newcommand{\KM}{\K--Moutard}

\numberwithin{equation}{section}

\theoremstyle{plain}
	\newtheorem{theorem}{Theorem}[section]
	\newtheorem{lemma}[theorem]{Lemma}
	\newtheorem{proposition}[theorem]{Proposition}
	\newtheorem{corollary}[theorem]{Corollary}

\theoremstyle{definition}
\newtheorem{definition}[theorem]{Definition}
\newtheorem{assumption}[theorem]{Assumption}
\newtheorem{notation}[theorem]{Notation}

\theoremstyle{remark}
\newtheorem{remark}[theorem]{Remark}
\newtheorem{remarks}[theorem]{Remarks}
\newtheorem{xmpl}[theorem]{Example}
\newtheorem{xmpls}[theorem]{Examples}

\makeatletter
\@namedef{subjclassname@2020}{\textup{2020} Mathematics
  Subject Classification}
\makeatother

\newcommand{\TitleWithUrl}[1]{\IfEmptyBibField{doi}%
  {\IfEmptyBibField{url}{\textit{#1}}%
    {\IfEmptyBibField{eprint}{\href {\BibField{url}}{\textit{#1}}}{\textit{#1}}}%
    }%
  {\href {https://doi.org/\BibField{doi}}{\textit{#1}}}}
\renewcommand{\eprint}[1]{\IfEmptyBibField{url}{\url{#1}}%
  {\href {\BibField{url}}{#1}}}

\BibSpec{article}{%
    +{}  {\PrintAuthors}                {author}
    +{,} { \TitleWithUrl}               {title}
    +{.} { }                            {part}
    +{:} { \textit}                     {subtitle}
    +{,} { \PrintContributions}         {contribution}
    +{.} { \PrintPartials}              {partial}
    +{,} { }                            {journal}
    +{}  { \textbf}                     {volume}
    +{}  { \PrintDatePV}                {date}
    +{,} { \issuetext}                  {number}
    +{,} { \eprintpages}                {pages}
    +{,} { }                            {status}
    +{,} { available at \eprint}        {eprint}
    +{}  { \parenthesize}               {language}
    +{}  { \PrintTranslation}           {translation}
    +{;} { \PrintReprint}               {reprint}
    +{.} { }                            {note}
    +{.} {}                             {transition}
    +{}  {\SentenceSpace \PrintReviews} {review}
}

\BibSpec{collection.article}{%
    +{}  {\PrintAuthors}                {author}
    +{,} { \TitleWithUrl}                     {title}
    +{.} { }                            {part}
    +{:} { \textit}                     {subtitle}
    +{,} { \PrintContributions}         {contribution}
    +{,} { \PrintConference}            {conference}
    +{}  {\PrintBook}                   {book}
    +{,} { }                            {booktitle}
    +{,} { \PrintDateB}                 {date}
    +{,} { pp.~}                        {pages}
    +{,} { }                            {status}
    +{,} { available at \eprint}        {eprint}
    +{}  { \parenthesize}               {language}
    +{}  { \PrintTranslation}           {translation}
    +{;} { \PrintReprint}               {reprint}
    +{.} { }                            {note}
    +{.} {}                             {transition}
    +{}  {\SentenceSpace \PrintReviews} {review}
}
\BibSpec{book}{%
    +{}  {\PrintPrimary}                {transition}
    +{,} { \TitleWithUrl}               {title}
    +{.} { }                            {part}
    +{:} { \textit}                     {subtitle}
    +{,} { \PrintEdition}               {edition}
    +{}  { \PrintEditorsB}              {editor}
    +{,} { \PrintTranslatorsC}          {translator}
    +{,} { \PrintContributions}         {contribution}
    +{,} { }                            {series}
    +{,} { \voltext}                    {volume}
    +{,} { }                            {publisher}
    +{,} { }                            {organization}
    +{,} { }                            {address}
    +{,} { \PrintDateB}                 {date}
    +{,} { }                            {status}
    +{}  { \parenthesize}               {language}
    +{}  { \PrintTranslation}           {translation}
    +{;} { \PrintReprint}               {reprint}
    +{.} { }                            {note}
    +{.} {}                             {transition}
    +{}  {\SentenceSpace \PrintReviews} {review}
}

\BibSpec{thesis}{%
    +{}  {\PrintAuthors}                {author}
    +{.} { \PrintDate}                  {date}
    +{.} { \TitleWithUrl}               {title}
    +{:} { \textit}                     {subtitle}
    +{,} { \PrintThesisType}            {type}
    +{,} { }                            {organization}
    +{,} { }                            {address}
    +{,} { \eprint}                     {eprint}
    +{,} { }                            {status}
    +{}  { \parenthesize}               {language}
    +{}  { \PrintTranslation}           {translation}
    +{;} { \PrintReprint}               {reprint}
    +{.} { }                            {note}
    +{.} {}                             {transition}
    +{}  {\SentenceSpace \PrintReviews} {review}
}

\title[Discrete $\Omega$-nets and Guichard nets via
Koenigs nets]{Discrete $\Omega$-nets and Guichard nets\\via
  discrete Koenigs nets}
\author{F.E. Burstall}
\address{Department of Mathematical Sciences\\ University of Bath\\
  Bath BA2 7AY\\UK} \email{feb@maths.bath.ac.uk}

\author{J. Cho}
\address{%
  Institute of Discrete Mathematics and Geometry\\ 
  TU Wien\\
  Wiedner Hauptstrasse 8-10/104\\
  1040 Wien\\
  Austria}
\email{joseph.cho@tuwien.ac.at}

\author{U. Hertrich-Jeromin}
\address{%
  Institute of Discrete Mathematics and Geometry\\ 
  TU Wien\\
  Wiedner Hauptstrasse 8-10/104\\
  1040 Wien\\
  Austria}
 \email{udo.hertrich-jeromin@tuwien.ac.at}

\author{M. Pember}
\address{Department of Mathematical Sciences\\ University of Bath\\
  Bath BA2 7AY\\UK}
\email{mason.j.w.pember@bath.edu}

\author{W. Rossman}
\address{Department of Mathematics, Graduate School of Science\\
  Kobe University \\ 1-1 Rokkodaicho Nada-ku Kobe 657-8501\\ Japan}
 \email{wayne@math.kobe-u.ac.jp}

\subjclass[2020]{Primary: 53A70; Secondary: 53A10, 53A31}

\begin{document}

\begin{abstract}
  We provide a convincing discretisation of Demoulin's
  $\Omega$-surfaces along with their specialisations to
  Guichard and isothermic surfaces with no loss of
  integrable structure.
\end{abstract}
\maketitle

\section{Introduction}

\subsection{Background}
\label{sec:background}

Our topic begins with the work of Darboux, Bianchi,
Guichard and Demoulin in the early 20th century.  This
period saw rapid progress in surface geometry of a kind that
we now recognise as a manifestation of the close relation of
geometry to soliton theory.  In particular, three
surface classes, of increasing generality, were introduced:
\emph{isothermic surfaces} by Bour \cite{bour_theorie_1862}
in 1862, \emph{Guichard surfaces} by Guichard
\cite{guichard_sur_1900} in 1900 and, finally,
\emph{$\Omega$-surfaces} by Demoulin
\cite{demoulin_sur_1911-2} in 1911.

These classes have a common formulation in terms of duality.
For this, let $x:\dom\to\R^{3}$ be an isometric immersion
of a surface and recall that a \emph{Combescure
  transformation of $x$} is a second immersion\footnote{The
  immersion requirement will be relaxed below.}
$\dual{x}:\dom\to\R^3$ with parallel curvature directions to
those of $x$.  Let $\kappa_{1},\kappa_2$ be the principal
curvatures of $x$ and $\dual{\kappa}_{1},\dual{\kappa}_2$
those of $\dual{x}$.  Now $x$ is isothermic if and only if
it has a Combescure transform $\dual{x}$, the
\emph{Christoffel dual} \cite{christoffel_ueber_1867}, for
which
\begin{subequations}\label{eq:61}
  \begin{equation}
    \label{eq:1}
    \frac1{\kappa_1\dual{\kappa}_2}+\frac1{\kappa_2\dual{\kappa}_1}=0.
  \end{equation}
  Meanwhile, Guichard's original definition of his eponymous
  surfaces is that there should be a Combescure transform
  $\dual{x}$, the \emph{associate surface}, such that
  \begin{equation}
    \label{eq:6}
    \frac1{\kappa_1\dual{\kappa}_2}+\frac1{\kappa_2\dual{\kappa}_1}=c\neq0,
  \end{equation}
  for some constant $c$.

  Finally, one of us \cite[Theorem~5.1]{pember_lie_nodate} observed that
  $\Omega$-surfaces may be similarly characterised in terms
  of two Combescure transforms $\dual{x}$ and $\dual{n}$,
  where the latter has principal curvatures
  $\ell_1,\ell_{2}$, for which
  \begin{equation}
    \label{eq:7}
    \frac1{\kappa_1\dual{\kappa}_2}+\frac1{\kappa_2\dual{\kappa}_1}
    =\frac1{\ell_1}+\frac1{\ell_2}.
  \end{equation}
  Observe that when $\dual{n}=n$, the Gauss map of $x$,
  \eqref{eq:7} reduces to \eqref{eq:6} so that Guichard
  surfaces are $\Omega$ as Demoulin observed
  \cite{demoulin_sur_1911}. Again, when $\dual{n}$ is
  constant, \eqref{eq:7} reduces to \eqref{eq:61} and
  isothermic surfaces are seen to be $\Omega$-surfaces also,
  another result of Demoulin \cite{demoulin_sur_1911-2}.

  In each case, $x$ and $\dual{x}$ appear symmetrically so
  that $\dual{x}$ is isothermic, Guichard or $\Omega$ as $x$
  is.
\end{subequations}

These characterisations admit a reformulation which is very
amenable to discretisation.  The equations \eqref{eq:61} are
equivalent to:
\begin{subequations}
  \label{eq:62}
  \begin{gather}
    \label{eq:63}
    \d x\curlywedge\d\dual{x}=0\\
    \label{eq:64}
    \d x\curlywedge\d\dual{x}+\d n\curlywedge \d n=0\\
    \label{eq:65}
    \d x\curlywedge\d\dual{x}+\d n\curlywedge \d\dual{n}=0.
  \end{gather}
\end{subequations}
Here $\curlywedge$ is exterior product of $\R^3$-valued
$1$-forms using the wedge product of $\R^3$ to multiply
coefficients so that \eqref{eq:62} are equations on
$\Wedge^2\R^3$-valued $2$-forms.  As we shall see in
Section~\ref{sec:o-systems}, equations \eqref{eq:62} show that
$\Omega$-surfaces are $O$-surfaces in the sense of
Schief--Konopelchenko \cite{SchKon03}.

On the down-side, these Euclidean formulations obscure the
symmetry of the situation: both isothermic and Guichard
surfaces are conformally invariant while $\Omega$-surfaces
are Lie sphere invariant.  Moreover, the fundamental role
played by isothermic surfaces is not apparent.  Demoulin's
original approach does not have these defects but requires a
change of viewpoint to that of Lie sphere geometry.  Here,
the basic idea is to study a surface via the collection of
$2$-spheres tangent to that surface.  The oriented
$2$-spheres in $S^3$ are parametrised by a $4$-dimensional
real quadric $\QQ$ in such a way that two such are in
oriented contact exactly when the corresponding points in
$\QQ$ lie on a (projective) line in $\QQ$.  Thus each line
in $\QQ$ parametrises the $1$-parameter family of oriented
$2$-spheres through a fixed point $p\in S^{3}$ and tangent
to a fixed $2$-plane in $T_pS^{3}$.  Otherwise said, the
contact elements of $S^3$ are parametrised by the
$5$-dimensional space $\cZ$ of lines in $\QQ$. One now
studies surfaces in $S^3$ by replacing them with their
Legendre lifts (their collection of contact elements), viewed
as maps into $\cZ$.  See \cite{cecil_lie_2008} for more
details.  With this in hand, Demoulin originally defined an
$\Omega$-surface to be a surface in $S^3$ whose Legendre
lift contains an isothermic surface in $\QQ$, otherwise
said, the surface admits an enveloping sphere congruence
which is isothermic \emph{qua} map into $\QQ$.  This is a
manifestly Lie sphere geometric characterisation of
$\Omega$-surfaces.  From here one can show that a surface is
$\Omega$ exactly when its Legendre lift is Lie applicable
\cite{Musso_2006} which gives a second Lie sphere geometric
characterisation.  Additionally, Demoulin shows there is a
second isothermic enveloping sphere congruence harmonically
separated from the first by the curvature spheres so that
$\Omega$-surfaces have Legendre lifts spanned by isothermic
sphere congruences.

In this invariant formulation of isothermic and $\Omega$
surfaces, the duality of \eqref{eq:62} manifests itself in
the existence of closed Lie algebra-valued $1$-forms whose
existence characterises the surfaces in question
\cite{burstall_isothermic_2011,Cla12,Musso_2006}.  These
$1$-forms are gauge potentials for pencils of flat
connections that provide an efficient way into the
integrable systems approach to these surfaces.

This integrability of all three surface classes is evidenced
by their rich transformation theory as developed by Bianchi,
Calapso and Darboux
\cite{bianchi_ricerche_1905,Bia05a,Cal03,Dar99e} for
isothermic surfaces and Eisenhart
\cite{Eis14,Eisenhart1915-fc,Eisenhart1916-qq} for Guichard
and $\Omega$-surfaces. In fact, the transformations of
Guichard and $\Omega$ surfaces are induced by the Darboux
and Calapso transforms of their isothermic sphere
congruences \cite{burstall_polynomial_2018}.  Indeed, even
the duality of $\Omega$-surfaces is induced by the
Christoffel duality of the sphere congruences\footnote{This
  is how Demoulin arrived at the dual $\Omega$-surface
  \cite{demoulin_sur_1911-1}.}.

Having reached an understanding of $\Omega$-surfaces, a
basic question is how to identify isothermic and Guichard
surfaces among them.  In the Euclidean formulation, this is
clear: one requires either that $\dual{n}$ be constant or
that $\dual{n}=n$, respectively.  To get an invariant
picture that is compatible with the transformation theory,
one can exploit the $1$-parameter family of flat connections
alluded to above.  Requiring that these connections admit
families of parallel sections depending affine linearly on the
parameter (\emph{linear conserved quantities}) leads to
such a characterisation of isothermic, Guichard and, indeed,
subclasses of isothermic surfaces
\cite{burstall_polynomial_2018}.

\subsection{Manifesto}
\label{sec:manifesto}

We have discussed the smooth theory of $\Omega$-surfaces and
their sub-classes at such length because the aim of this
paper is to provide a discrete theory that almost exactly
replicates the smooth one.  In particular, we shall define
discrete Guichard and $\Omega$-nets (extending the
well-known theory of isothermic nets along the way) as well
as discrete applicable Legendre maps so that:
\begin{compactitem}
\item Discrete isothermic, Guichard and $\Omega$-nets are
  characterised by \eqref{eq:62} and so have the same
  duality as the smooth case.  For this, we will need to
  develop a discrete exterior calculus.
\item $\Omega$-nets are discrete $O$-surfaces in the sense
  of Schief \cite{Sch03}.
\item Both isothermic nets and applicable Legendre maps are
  defined in terms of a closed Lie algebra valued $1$-form.
\item Applicable Legendre maps are generically spanned by
  isothermic sphere congruences.
\item A net is $\Omega$ if and only if it is enveloped by an
  isothermic sphere congruence so that its Legendre lift is
  applicable.
\item Isothermic and Guichard nets are characterised by the
  existence of linear conserved quantities just as in the
  smooth case \cite{burstall_polynomial_2018} as are
  $L$-isothermic and $L$-Guichard nets.
\item The transformation theory of isothermic nets induces a
  transformation theory of $\Omega$-nets, restricting to one
  of Guichard nets and the other subclasses.
\end{compactitem}

This match of smooth and discrete extends to fine detail.
For example, the radii of the isothermic spheres congruences
enveloping a Guichard net are reciprocal to those enveloping
the associate net: a result of Demoulin
\cite{demoulin_sur_1911-1} in the smooth case.  Again, a
classical formula of Eisenhart \cite{Eis14} relates
distances between corresponding points of a Darboux pair of
Guichard surfaces and their associates: this result holds
equally for Guichard nets and, in fact, we extend it to
$\Omega$-nets.

\subsection{Road map}
\label{sec:roadmap}

Let us briefly sketch the contents of the paper.

Section~\ref{sec:preliminaries} is mostly preparatory in
nature but has some novelty in that we develop a discrete
exterior calculus that we use extensively in the sequel.
Thus we define discrete exterior forms, their exterior
derivative and exterior product and show that these satisfy
the anti-commutativity and Leibniz rules familiar from the
smooth case.  Versions of discrete exterior calculus are
well-known in the finite elements literature
\cite{Hir03,DesHirLeoMar05,MR2488210,HydMan04} but the
technology is used more rarely in discrete differential
geometry.

Much of our paper is concerned with isothermic nets in the
Lie quadric which is a $4$-dimensional quadric with a
signature $(3,1)$ conformal structure.  The indefiniteness
of the conformal structure means much of the well-known
theory of isothermic nets in the definite case either needs
adjustment (the family of flat connections) or is not
available (cross-ratio factorising functions).  These
defects do not apply to the characterisation of
Bobenko--Suris
\cite{bobenko_discrete_2008,bobenko_discrete_2009} of
isothermic nets as \K\ nets (a purely projective notion)
taking values in a quadric.  We therefore start, in
Section~\ref{sec:k-nets-applicability}, by studying \K\ nets in a
projective space $\P(V)$.  We find a new invariant
characterisation of \K\ nets in terms of a closed
$\Wedge^{2}V$-valued $1$-form $\eta$ and then define
applicable line congruences similarly.  The main result is
that, in the presence of mild regularity assumptions, a line
congruence is applicable exactly when it is spanned by a
\KM\ pair of \K\ nets.

In Section~\ref{sec:isothermic-nets}, we apply this theory
to isothermic nets in an arbitrary non-singular quadric.  We
find that the $1$-form $\eta$ is exactly what is needed to
extend the key points of the theory to the indefinite case.
In particular, we use $\eta$ to construct the Christoffel
transform and the family of flat connections and, from them,
reach the transformation theory.  A distinguishing
characteristic of the indefinite case is the extension of
Darboux transforms to include the case of infinite spectral
parameter: it is precisely Darboux pairs of this kind that
span applicable line congruences whose lines lie in the
quadric.

Such line congruences are called \emph{applicable Legendre
  maps} and are the topic of section
\ref{sec:appl-legendre-maps}.  We find that the
transformation theory of the isothermic sphere congruences
induces one on the applicable Legendre maps they span
independently of choices.  In particular, we find a duality
of such line congruences.  Once more, the closed $1$-form
$\eta$ plays an essential role.

The remainder of the paper applies the preceding theory to
$\Omega$-nets (in Section~\ref{sec:discr-texorpdfstr-su})
and Guichard nets (in Section~\ref{sec:guichard-nets-via}).
The duality of applicable Legendre maps induces the duality
of $\Omega$-nets.  We prove that our notion of $\Omega$-net
coincides with that of \cite{burstall_discrete_2018} and
that Guichard nets, defined by $n=\dual{n}$, are also
characterised by polynomial conserved quantities.  We also
test the robustness of our discretisation by proving
discrete analogues of the classical results of Demoulin and
Eisenhart alluded to above.

Finally we make contact in Section~\ref{sec:o-systems-omega}
with Schief's theory of discrete $O$-surfaces for which our
exterior calculus furnishes an efficient methodology.  We
show that our notion of Guichard net coincides with that of
Schief and that $\Omega$-nets are $O$-surfaces.

\subsection{Acknowledgements}
\label{sec:acknowledgements}

The first author would like to thank David Calderbank for many
illuminating conversations on topics related to this work.
The third author thanks Wolfgang Schief for interesting and
pleasant conversations, in which an alternative viewpoint on
$\Omega$-surfaces as $O$-surfaces was independently developed.

Furthermore, we gratefully acknowledge financial support of the
project through the following research grants:
JSPS Grant-in-Aid
 for JSPS Fellows 19J10679,
 for scientific research
  (C) 15K04845, 20K03585, and
  (S) 17H06127 (P.I.: M.-H. Saito);
the JSPS/FWF Joint Project I3809-N32 ``Geometric shape generation'';
the FWF research project P28427-N35 ``Non-rigidity and symmetry breaking'';
and the MIUR grant ``Dipartimenti di Eccellenza'' 20182022,
 CUP: E11G18000350001, DISMA, Politecnico di Torino.

 Finally it is a pleasure to thank the referee for their
 careful reading and helpful comments.


\section{Preliminaries}
\label{sec:preliminaries}
Our approach to discrete differential geometry is to follow
the smooth situation as much as possible.  We will therefore
require both a discrete exterior calculus of forms and a
discrete theory of bundles and connections.  We begin by
sketching both of these.

First some notation: our domain will be a subset $\dom$ of
$\Z^{N}$.  The latter is organised into vertices, edges,
quadrilaterals (faces) and, more generally, $k$-cubes for
$0\leq k\leq n$.  We say that a $k$-cube belongs to $\Sigma$
if all of its $2^k$ vertices lie in $\dom$.

\subsection{Discrete exterior calculus}

\subsubsection{Discrete \texorpdfstring{$k$}{k}-forms}
\label{sec:discr-texorpdfkk-for}

\begin{definition}[$k$-form]
  A discrete $k$-form $\omega$ on $\dom$ is a real-valued
  function on the oriented $k$-cubes of $\dom$ with values
  that change sign when the orientation reverses:
  $\omega(-C)=-\omega(C)$.

  We denote the vector space of $k$-forms on $\dom$ by
  $\Omega^k_{\dom}$.
\end{definition}
Let us spell out what this means for the cases $k\leq 2$
which are of principal interest to us.
\begin{compactenum}
\item An orientation of a point is a sign so a $0$-form
  $f$ on $\dom$ satisfies $f(-p)=-f(p)$, for $p\in\dom$, and
  so amounts to a function $f:\dom\to\R$.
\item View an oriented edge as an ordered pair $ji$ of
  vertices, read as the edge from $i$ to $j$.  Now a
  $1$-form $\alpha$ on $\dom$ is a function on such pairs in
  $\dom$ with $\alpha_{ij} = - \alpha_{ji}$ (c.f.\
  \cite[Definition 2.23]{bobenko_discrete_2008}).
\item View an oriented quadrilateral as a cyclic ordering
  $\ijkl$ of its adjacent vertices so that a $2$-form
  $\omega$ on $\dom$ is a function of such cyclicly ordered
  quadrilaterals of $\dom$ with
  $\omega_{\ijkl}=-\omega_{ijk\ell}$.
\end{compactenum}

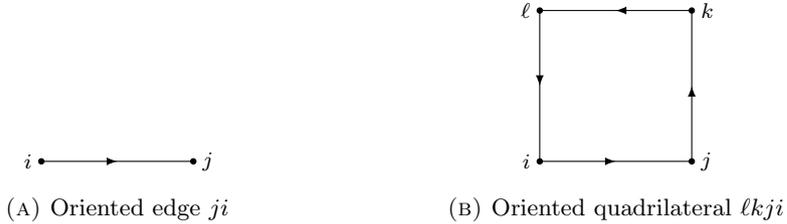
\begin{figure}[ht]
    \begin{subfigure}[t]{.4\linewidth}\centering
      \begin{tikzpicture}[x=1.0cm,y=1.0cm,decoration={
          markings, mark=at position 0.5 with
          {\arrow{latex}}},font=\footnotesize]
        \draw[black,postaction={decorate}] (0,0)--(2,0);
        \draw[fill] (0,0) node[left] {$i$}
        circle[radius=1pt]; \draw[fill] (2,0) node[right]
        {$j$} circle[radius=1pt];
      \end{tikzpicture}
      \caption{Oriented edge $ji$}
  \end{subfigure}
\hspace{.1\linewidth}
\begin{subfigure}[t]{.4\linewidth}\centering
  \begin{tikzpicture}[x=1.0cm,y=1.0cm,font=\footnotesize,decoration={
      markings, mark=at position 0.5 with {\arrow{latex}}}]
    \draw[black,postaction={decorate}] (0,0) --(2,0);
    \draw[black,postaction={decorate}] (2,0) --(2,2);
    \draw[black,postaction={decorate}] (2,2) --(0,2);
    \draw[black,postaction={decorate}] (0,2) --(0,0);
    \draw[fill] (0,0) node[left] {$i$} circle[radius=1pt];
    \draw[fill] (2,0) node[right] {$j$} circle[radius=1pt];
    \draw[fill] (0,2) node[left] {$\ell$}circle[radius=1pt];
    \draw[fill] (2,2) node[right] {$k$} circle[radius=1pt];
  \end{tikzpicture}
  \caption{Oriented quadrilateral $\ijkl$}
\end{subfigure}
  \caption{Edges and quadrilaterals: the labels are read
    from right to left.}
  \label{fig:1}
\end{figure}

\subsubsection{Exterior derivative}
\label{sec:exterior-derivative}

Just as in the smooth case, we have an exterior derivative
and exterior product which together satisfy Leibniz rule.
For our purposes, we only need this structure on
$\Omega^k_{\dom}$, for $k\leq 2$ so we restrict attention to
this case below.

\begin{definition}[Exterior derivative]\label{th:46}
  The \emph{exterior derivative}
  $\d:\Omega^k_{\dom}\to\Omega^{k+1}_{\dom}$, $k=0,1$ is defined by
  \begin{compactenum}
  \item $\d f_{ji}=f_j-f_i$, for $f\in\Omega^{0}_{\dom}$.
  \item
    $\d\alpha_{\ijkl}=\alpha_{i\ell}+\alpha_{\ell k}+\alpha_{kj}+\alpha_{ji}$, for $\alpha\in\Omega^1_{\dom}$.
  \end{compactenum}
\end{definition}

An easy computation gives:
\begin{proposition}
  \label{th:44}
  $\d\circ\d=0$.
\end{proposition}
\begin{remark}
  The exterior derivative can be extended to $k$-forms,
  $k\geq 2$, retaining $\d\circ\d=0$, but we will have no
  need of such generality here.
\end{remark}

\subsubsection{Exterior algebra}

\begin{definition}[Exterior product]\label{th:47}
  The \emph{exterior product} $\wedge:\Omega^k_{\dom}\times
  \Omega^{\ell}_{\dom}\to\Omega_{\dom}^{k+\ell}$ is defined, for $k+\ell\leq 2$, as
  follows:
  \begin{compactitem}
  \item For $f,g\in\Omega^0_{\dom}$, $(f\wedge g)_i=f_ig_i$.
  \item For $f\in\Omega^0_{\dom}$,
    $\alpha\in\Omega^1_{\dom}$,
    $(f\wedge\alpha)_{ij}=(\alpha\wedge f)_{ij}=\half(f_i+f_j)\alpha_{ij}$.
  \item For $f\in\Omega^0_{\dom}$,
    $\omega\in\Omega^2_{\dom}$,
    $(f\wedge\omega)_{\ijkl}=(\omega\wedge f)_{\ijkl}=\tfrac14(f_i+f_j+f_k+f_{\ell})\omega_{\ijkl}$.
  \item For $\alpha,\beta\in\Omega^1_{\dom}$,
    \begin{equation*}
      (\alpha\wedge\beta)_{\ijkl}=\tfrac14\bigl(
      (\alpha_{ji}
  + \alpha_{k\ell}) (\beta_{\ell i} + \beta_{kj})-
(\alpha_{\ell i} + \alpha_{kj}) (\beta_{ji}
  + \beta_{k\ell})
      \bigr).
    \end{equation*}
  \end{compactitem}
\end{definition}

The key point of these definitions is that the exterior
product retains the skew-symmetry and Leibniz rules of the
smooth setting. Indeed, a routine computation gives:
\begin{proposition}
  \label{th:45}
  Let $\alpha\in\Omega^k_{\dom}$ and
  $\beta\in\Omega^{\ell}_{\dom}$.  Then,
  \begin{compactenum}
  \item for $k+\ell\leq 2$,
    $\alpha\wedge\beta=(-1)^{k\ell}\beta\wedge\alpha$;
  \item For $k+\ell\leq 1$,
    $\d(\alpha\wedge\beta)=\d\alpha\wedge\beta+(-1)^k\alpha\wedge\d\beta$.
  \end{compactenum}
\end{proposition}

\begin{remark}
  One can extend the exterior product to arbitrary $k,\ell$
  in such a way that \cref{th:45} continues to hold.
  However, the product is not associative: we can have
  $f\wedge(g\wedge\alpha)\neq (f\wedge g)\wedge\alpha$ even
  for functions $f,g$ and a $1$-form $\alpha$.  It seems
  possible that $\d$ and $\wedge$ are the unary and binary
  operators for an $A_{\infty}$-algebra structure on
  $\Omega_{\dom}$: see \cite{MR2488210} for some evidence of
  this.
\end{remark}

\subsubsection{Vector-valued forms}
\label{sec:vector-valued-forms}

In what follows, we shall often have recourse to forms
taking values in a vector space $V$ rather than simply
$\R$.  We denote the space of $k$-forms on $\dom$ with values in $V$
by $\Omega^k_{\dom}(V)$.  Thus
$\Omega^k_{\dom}(V)=\Omega^k_{\dom}\tens_{\R}V$.

The exterior derivative
$\d:\Omega^k(V)\to\Omega^{k+1}(V)$ is defined just as in
\cref{th:46} but, for the exterior product, we need to be
able to multiply values which requires extra structure:
\begin{definition}[Exterior product of vector-valued forms]
  Let $V,W,U$ be vector spaces and $B:V\times W\to U$ a
  bilinear map.  Then there is a bilinear map
  \begin{align*}
    \Omega^k_{\dom}(V)\times\Omega_{\dom}^{\ell}(W)&\to\Omega_{\dom}^{k+\ell}(U)\\
    (\alpha,\beta)&\mapsto B(\alpha\wedge\beta)
  \end{align*}
  obtained by replacing all multiplications in the formulae
  of \cref{th:47} by applications of $B$.
\end{definition}
\begin{xmpl}
  For $\alpha\in\Omega_{\dom}^1(V)$,
  $\beta\in\Omega_{\dom}^1(W)$,
  $B(\alpha\wedge\beta)\in\Omega^2_{\dom}(U)$ is given by
  \begin{equation*}
    B(\alpha\wedge\beta)_{\ijkl}=\tfrac14\bigl(
B(\alpha_{ji}
  + \alpha_{k\ell},\beta_{\ell i} + \beta_{kj})-
B(\alpha_{\ell i} + \alpha_{kj},\beta_{ji}
  + \beta_{k\ell})      \bigr).
  \end{equation*}
\end{xmpl}

This product has Leibniz rule just as in \cref{th:45}(2)
and, when $V=W$, is graded-(anti)-commutative if $B$ is
(skew)-symmetric:
\begin{lemma}
  \label{th:48}
  Let $B:V\times V\to U$ be bilinear and
  $\alpha\in\Omega_{\dom}^k(V)$,
  $\beta\in\Omega_{\dom}^{\ell}(V)$ then:
  \begin{compactenum}[(a)]
  \item if $B$ is symmetric,
    $\alpha\wedge\beta=(-1)^{k\ell}\beta\wedge\alpha$;
  \item if $B$ is skew-symmetric,
    $\alpha\wedge\beta=(-1)^{k\ell+1}\beta\wedge\alpha$.
  \end{compactenum}
\end{lemma}

\begin{notation}\label{th:49}
  When $B:V\times V\to\Wedge^2V$ is exterior product
  $B(v,w)=v\wedge w$, we write $\alpha\curlywedge\beta$ for
  $B(\alpha\wedge\beta)$.

  In view of \cref{th:48}(b), we have
  $\alpha\curlywedge\beta=\beta\curlywedge\alpha$, for
  $\alpha,\beta\in\Omega^1_{\dom}(V)$.
\end{notation}

Our exterior calculus of vector valued forms gives a
convenient approach to the mixed area\footnote{In fact, the
  mixed area in \cite{bobenko_discrete_2008} is only defined
for polygons lying in translates of a fixed $2$-plane $U$
and so takes values in a fixed line $\Wedge^2U$ which is
identified with a scalar via a choice of area form.} discussed by
Bobenko--Suris \cite[\S4.5.2]{bobenko_discrete_2008}.  For
this, let $x:\dom\to V$ and contemplate an ordered
quadrilateral $\ijkl$ of $\dom$.  The area
$A(x,x)_{\ijkl}\in\Wedge^2V$ is given by
\begin{equation*}
  A(x,x)_{\ijkl}:=\half(x_i\wedge x_j+x_j\wedge x_k+x_k\wedge
  x_{\ell}+x_{\ell}\wedge x_i)=\half (x_i-x_k)\wedge(x_j-x_{\ell}).
\end{equation*}
However, we readily see that $x_i\wedge x_j=(x\curlywedge\d
x)_{ji}$ so the first equation reads
$A(x,x)=\half\d(x\curlywedge\d x)$ which is $\half\d
x\curlywedge \d x$ by the Leibniz rule. The mixed area is
obtained by polarising $A$ on the space of maps $y: \dom\to V$
edge-parallel to $x$ (that is, $\d y_{ji}\prl \d x_{ji}$ on all
edges $ji$) and we conclude:
\begin{lemma}
  \label{th:50}
  Let $x,y:\dom\to V$ be edge-parallel.  Then the mixed area
  of $x$ and $y$ is given by
  \begin{equation*}
    A(x,y)=\half \d x\curlywedge \d y.
  \end{equation*}
\end{lemma}

\subsection{Discrete gauge theory}
\label{sec:discr-gauge-theory}

The gauge theoretic approach to discrete differential
geometry is well-established, see, for example,
\cite{burstall_isothermic_2011,burstall_discrete_2014,burstall_discrete_2015}.
Here we simply recall the relevant definitions to set
notation:
\begin{definition}[Bundles and connections]
  A \emph{bundle} on $\dom$ is a surjection of a set $X$
  onto $\dom$ with fibres $X_i$, $i\in\dom$.

  A \emph{connection} $\Gamma$ on $X$ is an assignment of a
  bijection $\Gamma_{ji}:X_i\to X_j$ to each oriented edge
  $ji$ of $\dom$ such that $\Gamma_{ij}\Gamma_{ji}=1_{X_i}$
  on all edges.

  A \emph{gauge transformation} is a section $T$ of $\Aut(X)$
  the bundle of fibre automorphisms: thus a map $i\mapsto
  T_{i}\in\Aut(X_i)$.

  Gauge transformations act on connections by
  \begin{equation*}
    (T\cdot \Gamma)_{ji}=T_j\Gamma_{ji}T_i^{-1}.
  \end{equation*}

  A section $\sigma$ of $X$ is \emph{$\Gamma$-parallel} if, on each
  $ji$, we have
  \begin{equation*}
    \sigma_j=\Gamma_{ji}\sigma_{i}.
  \end{equation*}

  A connection $\Gamma$ is \emph{flat} if, on each
  quadrilateral $\ijkl$, we have
  \begin{equation*}
    \Gamma_{k\ell}\Gamma_{\ell i}=\Gamma_{kj}\Gamma_{ji},
  \end{equation*}
  a condition which is invariant under cyclic permutation of
  vertices.
  
  The action of gauge transformations permutes flat connections.
\end{definition}

\begin{remark}\label{th:54}
  In our applications, our bundles will have more structure:
  the fibres will be projective lines or vector spaces with
  an inner product and our connections and gauge
  transformations will preserve that structure.
\end{remark}

\subsection{Simply connected subsets of
  \texorpdfstring{$\Z^N$}{ZN}}
\label{sec:simply-conn-subs}

We want to integrate closed $1$-forms and flat connections
on $\dom$ which amounts to a topological requirement on
$\dom$.  Here are the ingredients:
\begin{definition}[Paths and homotopy]
  A \emph{path in $\dom$ from $p$ to $q$} is a
  sequence of vertices $(i_{n}\dots i_{0})$ in $\dom$ such
  that $i_0=p$, $i_n=q$ and each $i_{j+1}i_j$ is an edge in
  $\dom$.

  $\dom$ is \emph{connected} if there is a path between any
  two points of $\dom$.

  Two paths $\gamma,\hat{\gamma}$ in $\dom$ from $p$ to $q$
  are \emph{based homotopic}, $\gamma\simeq\hat{\gamma}$, if
  there is a sequence of paths $(\gamma_m)_{m=0}^l$ in
  $\dom$ from $p$ to $q$ with $\gamma_0=\gamma$,
  $\gamma_l=\hat{\gamma}$ and each $\gamma_m$ differs from
  $\gamma_{m+1}$ either:
  \begin{compactenum}[(a)]
  \item on a single edge by $(iji)$ (\cref{fig:iji}), or
  \item on a single quadrilateral $\ijkl$ by $(k\ell i)$
    versus $(kji)$ (\cref{fig:ijkl}).
  \end{compactenum}

  $\dom$ is \emph{simply connected} if it is connected and any
  two paths in $\dom$ with the same end-points are based
  homotopic.
\end{definition}
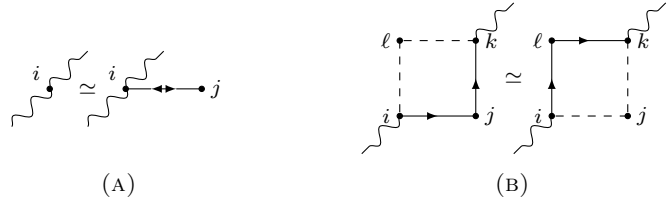
\begin{figure}[ht]
  \begin{subfigure}[b]{.4\linewidth}
    \centering
    \begin{tikzpicture}[x=.5cm,y=0.5cm,font=\footnotesize]
      \begin{scope}
        \draw[snake=snake] (-1,-1)--(1,1); \draw[fill] (0,0)
        node[above left] {$i$} circle[radius=1pt]; \draw
        (1,0) node {$\simeq$};
        \end{scope}
      \begin{scope}[xshift=1cm]
        \draw[black,] (0,0) --(.67,0);
        \draw[black,latex-latex] (.67,0)-- (1.33,0); \draw
        (1.33,0) -- (2,0); \draw[snake=snake]
        (-1,-1)--(1,1); \draw[fill] (0,0) node[above left]
        {$i$} circle[radius=1pt]; \draw[fill] (2,0)
        node[right] {$j$} circle[radius=1pt];
        \draw (0,-1.5) node{};
      \end{scope}
    \end{tikzpicture}
    \caption{}\label{fig:iji}
  \end{subfigure}
  \begin{subfigure}[b]{.4\linewidth}
    \centering
    \begin{tikzpicture}[x=.5cm,y=0.5cm,font=\footnotesize,decoration={
      markings, mark=at position 0.5 with {\arrow{latex}}}]
    \draw[black,postaction={decorate}] (0,0)--(2,0);
    
  \draw[black,postaction={decorate}] (2,0) --(2,2);
  \draw[dashed] (2,2) --(0,2);
  \draw[dashed] (0,2) --(0,0);
    \draw[fill] (0,0) node[left] {$i$} circle[radius=1pt];
    \draw[fill] (2,0) node[right] {$j$} circle[radius=1pt];
    \draw[fill] (0,2) node[left] {$\ell$}circle[radius=1pt];
    \draw[fill] (2,2) node[right] {$k$} circle[radius=1pt];
    \draw[snake=snake] (2,2)--(3,3);
    \draw[snake=snake](0,0)--(-1,-1);
    \draw (3,1) node {$\simeq$};
    \begin{scope}[xshift=2cm]
      \draw[black,postaction={decorate}] (0,0)--(0,2);
    
  \draw[black,postaction={decorate}] (0,2) --(2,2);
  \draw[dashed] (2,2) --(2,0);
  \draw[dashed] (2,0) --(0,0);
    \draw[fill] (0,0) node[left] {$i$} circle[radius=1pt];
    \draw[fill] (2,0) node[right] {$j$} circle[radius=1pt];
    \draw[fill] (0,2) node[left] {$\ell$}circle[radius=1pt];
    \draw[fill] (2,2) node[right] {$k$} circle[radius=1pt];
    \draw[snake=snake] (2,2)--(3,3);
    \draw[snake=snake](0,0)--(-1,-1);
    \end{scope}
  \end{tikzpicture}
  \caption{}\label{fig:ijkl}
  \end{subfigure}
  \caption{Elementary homotopies}
  \label{fig:2}
\end{figure}

\begin{xmpl}
  Given a quadrilateral $\ijkl$, we have
  $(ji)\simeq (jk\ell i)$.  Indeed,
  $(jk\ell i)\simeq (jkji)$ by (b) and then
  $(jkji)\simeq (ji)$ by (a).
\end{xmpl}

A pleasant exercise provides sufficient examples for our
needs:
\begin{proposition}
  \label{th:51}
\item[]
  \begin{compactenum}[(a)]
  \item $\Z^N$ is simply connected.
  \item If $\dom\subset\Z^N$ is simply connected so is
    $\dom\times\set{0,1}\subset\Z^{N+1}$.
  \end{compactenum}
\end{proposition}

We now have:
\begin{proposition}
  \label{th:52}
  Let $\dom$ be simply connected.  Then
  \begin{compactenum}[(a)]
  \item Closed $1$-forms are exact:
    $\alpha\in\Omega^1_{\dom}(V)$ has $\d\alpha=0$ if and only
    if $\alpha=\d f$, for some $f\in\Omega^0_{\dom}(V)$, unique up
    to addition of a constant.
  \item Flat connections are trivialisable: if $\Gamma$ is a
    flat connection on a bundle $X\to\dom$ and
    $p\in\dom$ is fixed, there are isomorphisms
    $T_i:X_i\to X_p$, $i\in\dom$, such that
    $\Gamma_{ji}=T_j^{-1}T_i$, on all edges $ji$ of $\dom$.

    Equivalently, there is a $\Gamma$-parallel section of
    $X$ through any point of $X_p$.
  \end{compactenum}
\end{proposition}
\begin{proof}
  For (a), let $\alpha$ be a closed $1$-form, choose $f(p)$
  arbitrarily and define $f(q)$ by
  \begin{equation*}
    f(q)=\sum_{j=0}^{n-1}\alpha_{i_{j+1}i_j},
  \end{equation*}
  for some path $i_n\dots i_0$ from $p$ to $q$.  We note
  that, $\alpha$ being a closed $1$-form, we have
  \begin{equation*}
    \alpha_{ij}+\alpha_{ji}=0\qquad
    \alpha_{k\ell}+\alpha_{\ell i}=\alpha_{kj}+\alpha_{ji}
  \end{equation*}
  so that $f$ is well-defined since $\dom$ is simply
  connected.

  The proof of (b) is similar: define $T_i$ by
  \begin{equation*}
    T_i=\Gamma_{i_{0}i_i}\cdots\Gamma_{i_{n-1}i_n},
  \end{equation*}
  for some path $i_0\dots i_n$ from $p$ to $i$.  This is
  well-defined as before, since $\Gamma$ is flat and $\dom$
  is simply connected.

  Finally, for $\sigma_p\in X_p$, $\sigma:=T^{-1}\sigma_p$
  is a $\Gamma$-parallel section of $X$ while, conversely,
  given sufficiently many parallel sections, define
  $T_i:\sigma_i\mapsto\sigma_p$ to construct $T$.
\end{proof}


\section{\K\ nets and applicability}
\label{sec:k-nets-applicability}
\subsection{\K\ nets}
\label{sec:k-nets}

Let $V$ be a vector space and let
$s : \dom \to \mathbb{P}(V)$ be a map or, equivalently, a
line subbundle $s\leq \dom\times V$, on which we impose,
without further comment, the following regularity
conditions:
\begin{assumption}\label{th:19}
  On each quadrilateral $\ijkl$:
  \begin{compactenum}
  \item $s_i,s_j,s_k,s_{\ell}$ are pair-wise distinct;
  \item $s_i,s_j,s_k,s_{\ell}$ are not collinear:
    $\dim \{s_i + s_j + s_k + s_\ell\} \geq 3$.
  \end{compactenum}
\end{assumption}

\begin{definition}[\K\ net]\label{def:isothermic}
  A map $s : \dom \to \mathbb{P}(V)$ is \emph{\K} if there
  exists a $\Wedge^2 V$-valued never-zero $1$-form
  $\eta\in\Omega^1_{\dom}(\Wedge^2V)$ such that
  \begin{compactenum}
  \item $\eta_{ji} \in s_j \wedge s_i\leq \Wedge^2V$, for each edge $ij$;
  \item $\eta$ is closed: $\d\eta=0$.
  \end{compactenum}
  Remark that if $\eta$ satisfies these conditions, so does
  any non-zero constant scale of $\eta$.

  We abuse notation and write $(s,\eta):\dom\to\P(V)$ for
  the package of a \K-net with a particular choice of
  $\eta$.
\end{definition}

Our first order of business is to reconcile this
projectively invariant notion with the affine-geometric one
of Bobenko--Suris
\cite{bobenko_discrete_2008,bobenko_discrete_2009}.  For the
latter, we recall:
\begin{definition}[\K\ dual]
  Let $F:\dom\to\A$ be a map to an affine space $\A$.  A
  \emph{\K\ dual} to $F$ is a map $\dual{F}:\dom\to \A$
  which is
  \begin{compactenum}
  \item edge-parallel to $F$:
    $\d F_{ij}\prl\d\dual{F}_{ij}$ on each edge $ij$;
  \item has parallel opposite diagonals to $F$: for each
    quadrilateral $\ijkl$,
    $F_k-F_i\prl \dual{F}_{\ell}-\dual{F}_j$ and
    $F_\ell-F_{j}\prl \dual{F}_k-\dual{F}_i$, or,
    equivalently \cite[Theorem 4.42]{bobenko_discrete_2008},
    $A(F,\dual{F})=0$.
  \end{compactenum}
\end{definition}

For Bobenko--Suris, $s$ is \K\ when its image in some affine
chart has a \K\ dual.  This coincides with our notion thanks
to the following:
\begin{proposition}
  \label{th:1}
  $s:\dom\to\P(V)$ is \K\ if and only if $s$ has an
  affine lift $F\in\Gamma s$ which admits a \K\ dual
  $\dual{F}$.

  In this case, we may take
  \begin{equation}
    \label{eq:3}
    \eta=\d\dual{F}\curlywedge F.
  \end{equation}
\end{proposition}
\begin{proof}
  Let $\alpha\in V^{*}$ define a hyperplane
  $H=\P(\ker\alpha)\leq\P(V)$ which we may assume is
  disjoint from the image of $s$.  Then $\alpha$ determines
  an affine lift $F\in\Gamma s$ with $\alpha(F)\equiv -1$ so
  that $F:\dom\to\A:=\set{v\in V\st\alpha(v)=-1}$.

  Now suppose that $(s,\eta)$ is \K\ and contemplate the
  interior product $i_{\alpha}\eta\in\Omega^1(\ker\alpha)$
  which is closed, since $\eta$ is, and so of the form
  $\d\dual{F}$ for some $\dual{F}:\dom\to\A$.  We claim
  that $\dual{F}$ is \K\ dual to $F$.

  For this, observe that on an edge $ij$,
  $\eta_{ji}=\lambda_{ji}F_j\wedge F_i$, $\lambda_{ij}\in\R$
  so that
  \begin{equation*}
    \d\dual{F}_{ji}=i_{\alpha}\eta_{ji}=\lambda_{ji}(F_j-F_i)
    =\lambda_{ji}\d F_{ji}.
  \end{equation*}
  Thus we see first that $\dual{F}$ is edge-parallel to $F$
  and then that $\eta=\d\dual{F}\curlywedge F$.  Taking
  the exterior derivative and using the Leibniz rule then
  yields
  \begin{equation*}
    \d\dual{F}\curlywedge\d F=0,
  \end{equation*}
  or, equivalently by \cref{th:50}, $A(F,\dual{F})=0$ which
  settles the claim.

  For the converse, given an affine lift $F$ of $s$ with \K\
  dual $\dual{F}$, define $\eta$ by \eqref{eq:3}.  Then
  $\eta_{ji}\in s_i\wedge s_j$, since $\d\dual{F}_{ji}$ is
  parallel to $\d F_{ji}$, while $\eta$ is closed since
  $A(F,\dual{F})=0$.
\end{proof}

Our projectively invariant formulation of the \K\ condition
gives us a fast proof of a third characterisation, due to
Bobenko--Suris, of \K\ nets via Moutard lifts.  For this, we
need a simple application of Cartan's Lemma:

\begin{lemma}\label{lem:r1}
  Let $a,b,c,d \in V$ and $r \in \R$ such that
  \[
    a \wedge (c-rb) = d \wedge (c-b).
  \]
  Then either $r = 1$ or
  $\dim \spn{a, b, c, d} \leq 2$, where, here and below,
  $\spn{\cdot}$ denotes linear span of vectors.
\end{lemma}

\begin{proof}
  Suppose that $\dim \spn{a, b, c, d} \geq 3$.  Then at
  least one of $a \wedge d$ and $b \wedge c$ is non-zero.
  If $a \wedge d \neq 0$, Cartan's Lemma tells us that
  $c-b, c-rb \in U := \spn{a, d}$, and, unless
  $r = 1$, this gives $(1-r)b \in U$ and thus $c \in U$: a
  contradiction.
	
  If $a \wedge d = 0$ then $b \wedge c \neq 0$, and so,
  unless $r=1$, $(c-rb)\wedge(c-b)\neq 0$.  Now Cartan's
  Lemma tells us that $a,d\in\spn{c-rb,c-b}=\spn{b,c}$
  which is again a contradiction.
\end{proof}

With this in hand, we have:
\begin{theorem}[{\cite[Theorem
    2.32]{bobenko_discrete_2008}}]
  \label{thm:isothermicMoutard}
  $s$ is \K\ if and only if there exists
  $\mu \in \Gamma s^{\times}$ satisfying the Moutard
  equation
  \begin{equation}
    \label{eq:5}
    \d\mu\curlywedge\d\mu=0
  \end{equation}
  or, in more familiar terms,
  \begin{equation}\label{eqn:Moutard}
    (\mu_k - \mu_i) \wedge (\mu_\ell - \mu_j) = 0,
  \end{equation}
  on each quadrilateral $\ijkl$.

  In this case, $\eta$ can be taken to be
  $\d\mu\curlywedge\mu$ so that, on each edge $ij$,
  \begin{equation}
    \label{eq:4}
    \eta_{ji}=\mu_j\wedge\mu_i.
  \end{equation}
  We call $\mu$ a \emph{Moutard lift of $s$}.
\end{theorem}

\begin{proof}
  First assume that $(s, \eta)$ is \K\ so that $\eta$ is a
  closed $1$-form with $\eta_{ji} \in s_j \wedge s_i$ on
  each edge $ij$.  Given $\mu_i \in s_i$, define
  $\mu_j \in s_j$ to ensure that \eqref{eq:4} holds:
  $\eta_{ji} = \mu_j \wedge \mu_i$.  To see that this is
  well-defined, consider a quadrilateral
  $\ijkl$: starting with $\mu_i$, we get $\mu_j \in s_j$
  and $\mu_\ell \in s_\ell$ with
  \[
    \eta_{ji} = \mu_j \wedge \mu_i, \quad \eta_{\ell i} =
    \mu_\ell \wedge \mu_i,
  \]
  and then $\mu'_k, \mu''_k \in s_k$ with
  \[
    \eta_{kj} = \mu'_k \wedge \mu_j, \quad \eta_{k \ell} =
    \mu''_k \wedge \mu_\ell.
  \]
  The closedness of $\eta$ now reads
  \[
    \mu_j \wedge (\mu_i-\mu'_k) = \mu_\ell \wedge (\mu_i-\mu''_k).
  \]
  Write $\mu''_k = r\mu'_k$ for some $r \in \mathbb{R}$ and
  apply Lemma \ref{lem:r1} to see that $r = 1$ since
  $\dim \spn{\mu_i, \mu_j, \mu'_k, \mu_\ell} \geq 3$ by
  \cref{th:19}.  Thus $\mu'_k = \mu''_k$,
  and so $\mu$ is well-defined and
  $\eta=\d\mu\curlywedge\mu$.  Now
  \begin{equation*}
    0=\d\eta=-\d\mu\curlywedge\d\mu
  \end{equation*}
  so that $\mu$ is a Moutard lift.
		
  Conversely, if $\mu \in \Gamma s^{\times}$ is a Moutard
  lift, set $\eta=\d\mu\curlywedge\mu$ to conclude
  that $(s,\eta)$ is \K.
\end{proof}

\begin{remark}
  We see at once from \eqref{eqn:Moutard} that \K\ nets have
  planar quadrilaterals and so are $Q$-nets in the sense of
  Bobenko--Suris \cite[Definition~2.1]{bobenko_discrete_2008}.
\end{remark}

A central tenet of the Bobenko--Suris philosophy
\cite{bobenko_discrete_2008} is that transformations are the
same as higher-dimensional nets.  With this in mind, we
introduce the following:
\begin{notation}
  For maps $x^{\pm}:\dom\to X$, a set, we define $x^+\sqcup
  x^-:\set{0,1}\times\dom\to X$ by
  \begin{equation*}
    (x^+\sqcup x^-)\restr{\set{0}\times\dom}=x^+,\qquad
    (x^+\sqcup x^-)\restr{\set{1}\times\dom}=x^-.
  \end{equation*}
  We visualise $\set{0,1}\times\dom\subset\Z^{N+1}$ as two
  copies of $\dom$ stacked on top of each other and refer to
  edges $\set{0,1}\times\set{i}$ and quadrilaterals
  $\set{0,1}\times\set{i,j}$ as \emph{vertical}.
\end{notation}

We now have:
\begin{definition}[\KM\ transformation]
  Let $(s^\pm, \eta^\pm):\dom\to\P(V)$ be two \K\ nets and
  let $s=s^+\sqcup s^-:\set{0,1}\times\dom\to\P(V)$.

  Assume that $s$ is regular so that $s^{\pm}_i,s^{\pm}_j$
  are all distinct and
  $\dim s^+_i + s^+_j + s^-_j + s^-_i \geq 3$.

  We say that $s^-$ is a \emph{\KM\ transformation} of $s^+$,
  or that $s^+,s^-$ are a \emph{\KM\ pair} if $(s,\eta)$ is
  also \K\ with $\eta$ satisfying
  \begin{equation*}
    \eta\restr{\{0\} \times \dom} = \eta^+,\qquad
    \eta\restr{\{1\} \times \dom} = \eta^-.
  \end{equation*}
\end{definition}

For a more practical formulation, define
$\tau\in\Gamma(s^+\wedge s^-)$ by
\begin{equation*}
  \tau_i=\eta_{(1,i)(0,i)}
\end{equation*}
and use the closedness of $\eta$ on vertical quadrilaterals
to conclude:
\begin{proposition}
  \label{th:2}
  \K\ nets $(s^{\pm},\eta^{\pm})$ are a \KM\ pair if and
  only if there is a section $\tau$ of $s^+\wedge s^-$ such
  that
  \begin{equation}\label{eqn:K-M}
    \eta^- = \eta^+ + \d \tau.
  \end{equation}
\end{proposition}
Alternatively \cref{thm:isothermicMoutard} implies the
following:
\begin{corollary}\label{cor:K-Mconsistency}
  $(s^\pm, \eta^\pm)$ are a \KM\ pair if and only if there
  are Moutard lifts $\mu^\pm \in \Gamma s^\pm$ such that
  \begin{equation}\label{eqn:MoutardOnTheSides}
    (\mu^+_j - \mu^-_i) \wedge (\mu^+_i - \mu^-_j) = 0.
  \end{equation}

  In this case, we may take
  $\eta^{\pm}=\d\mu^{\pm}\curlywedge\mu^{\pm}$ and then
  the section $\tau\in\Gamma(s^+\wedge s^{-})$ of
  \cref{th:2} is $\mu^-\curlywedge \mu^+$.
\end{corollary}

\begin{remark}
  \Cref{cor:K-Mconsistency} says that two \K\ nets $s^\pm$
  are a \KM\ pair if and only if there are Moutard lifts
  $\mu^\pm \in \Gamma s^\pm$ so that $\mu^\pm$ are Moutard
  transformations \cite[Definition
  2.36]{bobenko_discrete_2008} of each other.  Therefore,
  \cite[Theorem 2.7]{bobenko_isothermic_2007} implies that
  the \KM\ transformation is three-dimensionally consistent,
  and hence multi-dimensionally consistent.
\end{remark}

\subsection{Applicable line congruences}

Let $G_2(V)$ be the Grassmannian of $2$-planes in $V$ or,
equivalently, lines in $\P(V)$.
\begin{definition}[Line congruence {\cite[Definition~2.1]{MR1737004}}]
  A map $f:\dom\to G_2(V)$ is a \emph{line congruence} if,
  on each edge $ij$, $f_i\cap f_j\neq\set{0}$.
\end{definition}

For a \KM\ pair $(s^+,s^-)$ of \K\ nets, contemplate the map
$f=s^+\oplus s^-:\dom\to G_2(V)$.  In view of
\eqref{eqn:MoutardOnTheSides}, we see that, on each edge
$ij$, $\dim f_i+f_j\leq 3$ so that $f$ is a line congruence.

Line congruences of this kind are central to our programme
and, in this section, we characterise them in terms of
closed $\Wedge^2V$-valued $1$-forms parallel to our
definition of \K\ nets.

So let $f : \dom \to G_2(V)$ be a discrete line congruence
in $\mathbb{P}(V)$ and impose the following regularity
conditions:
\begin{assumption}\label{th:20}\item[]
  \begin{compactenum}
  \item First order: on each edge $ij$, $f_{ij}:=f_i+f_j$ is
    $3$-dimensional so that $s_{ij} := f_i \cap f_j$ has
    $\dim s_{ij} = 1$.
  \item Second order: on each quadrilateral $\ijkl$, we
    have
    \begin{equation*}
      s_{ij} \wedge s_{jk} \wedge s_{k\ell}
      \wedge s_{\ell i}\neq\set{0}.
    \end{equation*}
    It then follows that
    \begin{equation*}
      f_i=s_{\ell i}\oplus s_{ij},\qquad
      f_{ij}=s_{\ell i}\oplus s_{ij}\oplus s_{jk}
    \end{equation*}
    and cyclic permutations of these.
  \end{compactenum}
\end{assumption}
\begin{remark}\label{th:15}
  In this case, as we will see in \cref{th:3} below, a line
  subbundle $s<f$ is regular in the sense of
  \cref{th:19} so long as $s_i\neq s_{ij}$ on any edge
  $ij$.
\end{remark}

\begin{definition}[Applicable line congruence]
  We say that $f$ is \emph{applicable} if there is a
  $\Wedge^2 V$-valued $1$-form
  $\eta\in\Omega^1_{\dom}(\Wedge^2V)$ with (notation as in \cref{th:20}):
  \begin{enumerate}
  \item $\eta$ is closed.
  \item $\eta_{ij} \in \Wedge^2 f_{ij} = f_j \wedge f_i$.
  \item (non-degeneracy)
    $\eta_{ij} \wedge s_{ij} \neq \{0\}$, on each edge $ij$.
  \end{enumerate}
\end{definition}

\begin{xmpl}\label{h:appl-line-congr}
  If $(s,\eta)$ is a \K\ net with $s<f$ then $f$ is
  applicable via the same $\eta$ since each
  $s_i\wedge s_j< \wedge^{2}f_{ij}$.
\end{xmpl}

There is a gauge freedom in the choice of $\eta$: if $\eta$
satisfies the conditions of the definition, so does
$\eta^\tau := \eta + \d \tau$ for any section $\tau$ of
$\Wedge^2 f$.  Indeed $\eta^\tau$ is certainly closed;
$\eta^\tau_{ji} = \eta_{ji} + \tau_j -
\tau_i\in\Wedge^2f_{ij}$ since
$\tau_j, \tau_i \in \Wedge^2 f_{ij}$ while, for the
non-degeneracy, note that
$\tau_i \wedge s_{ij} = \tau_j \wedge s_{ij} = 0$.  We
denote by $[\eta]$ the equivalence class of all $1$-forms
that arise this way:
\begin{equation*}
  [\eta]=\set{\eta+\d\tau\st \tau\in\Gamma\Wedge^2f}.
\end{equation*}

\begin{remark}
  It would be interesting to know under what circumstances
  $f$ is applicable with respect to $\eta_1$ and $\eta_2$
  with $[\eta_1]\neq [\eta_2]$.  See, for example,
  Musso--Nicolodi \cite{Musso_2006} for the smooth case. 
\end{remark}

We are going to show in \cref{th:4} that, up to gauge, all
applicable line congruences $(f,[\eta])$ arise from a \K\
net as in \cref{h:appl-line-congr}.  This will require some
preparation.

\subsubsection{Flat connection of an applicable net}

We begin by discussing the projective geometry of
$\eta_{ji}$ on a single edge.  Since $\dim f_{ij} = 3$,
$\eta_{ji}$ is decomposable and so determines a $2$-plane in
$f_{ij}$ or, equivalently, a projective line in the
projective plane $\mathbb{P}(f_{ij})$.  We denote this
$2$-plane or projective line by $\abrack{\eta_{ji}}$, that
is, $\abrack{a \wedge b} = \spn{a, b}$.  The non-degeneracy
condition tells us $s_{ij}$ does not lie on
$\abrack{\eta_{ji}}$ so that $\abrack{\eta_{ji}}$ is
distinct from both of the projective lines $\mathbb{P}(f_i)$
and $\mathbb{P}(f_j)$.  We can therefore define
$s_i^{ij} \in \mathbb{P}(f_i)$ and
$s_j^{ij} \in \mathbb{P}(f_j)$ by
\[
  s_i^{ij} = \abrack{\eta_{ji}} \cap \mathbb{P}(f_i), \quad
  s_j^{ij} = \abrack{\eta_{ji}} \cap \mathbb{P}(f_j).
\]
Then $\eta_{ji} \in s_i^{ij} \wedge s_j^{ij}$ and
$s_i^{ij} \wedge s_{ij} \wedge s_j^{ij} \neq \{0\}$.

For $\tau_j \in \Wedge^2 f_j$, $r \in \mathbb{R}$, not both
zero, $\abrack{r \eta_{ji} + \tau_j}$ is a line in the pencil
through $\abrack{\eta_{ji}}$ and $\mathbb{P}(f_j)$.  Thus
replacing $\eta_{ji}$ with $r \eta_{ji} + \tau_j$ changes
the intersection with $\mathbb{P}(f_i)$ and leaves that with
$\mathbb{P}(f_j)$ untouched.  In particular, define a map
$g_{ij}: \mathbb{P}(\Wedge^2 f_j \oplus \mathbb{R}) \to
\mathbb{P}(f_i)$ by
\[
  g_{ij}([\tau_j, r]) = \abrack{r \eta_{ji} + \tau_j} \cap
  \mathbb{P}(f_i) \in \mathbb{P}(f_i).
\]
The situation is illustrated in Figure~\ref{fig:g_ij}.
\begin{figure}
  \centering
  \begin{tikzpicture}[x=1.0cm,y=.8cm]
    \draw
    (-2.2744827586206897,5.186206896551725)--(0.3158620689655174,-1.2896551724137932)node[right]{$\P(f_{i})$};
    \draw (-0.43692307692307697,-1.155384615384616)--
    node[right]
    {$\,\P(f_{j})=\abrack{\tau_{j}}$}(3.563076923076923,4.844615384615385);
    \draw (3.8534653465346533,3.9146534653465346)--
    node[above] {$\abrack{\eta_{ij}}$}
    (-2.6831683168316833,4.568316831683168); \draw [dash
    pattern=on 5pt off 5pt] (3.712,4.356)-- node[above left]
    {$\abrack{r\eta_{ij}+\tau_{j}}$} (-1.728,1.636); \draw
    [fill=black] (-2.,4.5) circle (1.5pt) node[below
    left]{$s^{ij}_{i}$}; \draw [fill=black] (0.,-0.5) circle
    (1.5pt) node[left]{$s_{ij}$}; \draw [fill=black] (3.,4.)
    circle (1.5pt) node[below right]{$s^{ij}_{j}$}; \draw
    [fill=black] (-1.,2.) circle (1.5pt) node[above
    left]{$g_{ji}([\tau_j,r])\,\,$};
  \end{tikzpicture}
  \caption{}\label{fig:g_ij}
\end{figure}

From elementary projective geometry, we have:
\begin{lemma}\label{lem:gijisomorphism}
  $g_{ij}$ is an isomorphism of projective lines with
  $g_{ij}([\tau_j, 0]) = s_{ij}$.
\end{lemma}

Now for the main construction of this section: write
$\dom = \dom_b \sqcup \dom_w$ in such a way that each edge
has one vertex in $\dom_b$ and one in $\dom_{w}$ and define
two bundles of projective lines $X^b, X^w$ over $\dom$ as
follows:
\begin{align*}
  X^b\restr{\dom_b} &:=
                      \mathbb{P}(f)\restr{\dom_b} 
  & X^b\restr{\dom_w} &:= \mathbb{P}(\Wedge^2 f \oplus \mathbb{R})\restr{\dom_w}\\
  X^w\restr{\dom_b} &:= \mathbb{P}(\Wedge^2 f \oplus
                      \mathbb{R})\restr{\dom_b} 
  & X^w\restr{\dom_w} &:= \mathbb{P}(f)\restr{\dom_w}.
\end{align*}
Now use the $g_{ij}$ and their inverses to define
connections $\gamma^b, \gamma^w$ on
the projective line bundles
$X^b, X^w$ (c.f. \cref{th:54}).  In detail,
\[
  \gamma^b_{ij} :=
  \begin{cases}
    g_{ij}, & i \in \dom_b\\
    g^{-1}_{ji}, & i \in \dom_w
  \end{cases}, \quad \gamma^w_{ij} :=
  \begin{cases}
    g_{ij}, & i \in \dom_w\\
    g^{-1}_{ji}, & i \in \dom_b
  \end{cases}.
\]

The key point is that these connections are flat:
\begin{proposition}
  On any quadrilateral $\ijkl$, we have
  \[
    \gamma^w_{ij}\gamma^w_{jk}\gamma^w_{k
      \ell}\gamma^w_{\ell i} = \mathrm{id}_{X_i} =
    \gamma^b_{ij}\gamma^b_{jk}\gamma^b_{k
      \ell}\gamma^b_{\ell i}.
  \]
\end{proposition}
\begin{proof}
  For $(X, \gamma)$ one of the bundles with connection under
  consideration, cyclically permute the indices if necessary
  to arrange that $X_i = \mathbb{P}(f_i)$.  Choose
  $s_i \in \mathbb{P}(f_i)$ with
  $s_i \neq s_{ij}, s_{i\ell}$.  By continuity, it suffices
  to prove that
  \[
    (\gamma_{k\ell} \circ \gamma_{\ell i})(s_i) =
    (\gamma_{kj} \circ \gamma_{ji})(s_i).
  \]
	
  Call the left side $s'_k$ and the right $s''_k$.  We want
  to prove that these coincide so suppose, for a
  contradiction, that they do not.  Then, with
  $\gamma_{\ell i}(s_i) = [\tau_\ell ,1]$ and
  $\gamma_{ji}(s_i) = [\tau_j, 1]$, we have
  \begin{equation}\label{eqn:flat1}
    (\eta_{\ell i} + \tau_\ell), (\eta_{ji} + \tau_j) \in s_i \wedge V,
  \end{equation}
  while there are
  $s_j=s^{kj}_{j} \in \mathbb{P}(f_j) \setminus s_{kj}$,
  $s_\ell=s^{k\ell}_{\ell} \in \mathbb{P}(f_\ell) \setminus
  s_{k\ell}$ such that
  \[
    \eta_{\ell k} + \tau_\ell \in s_\ell \wedge s_k', \quad
    \eta_{jk} + \tau_j \in s_j \wedge s_k''.
  \]
  Since $\eta$ is closed we also have
  \begin{equation}\label{eqn:flat2}
    (\eta_{\ell i} + \tau_\ell) - (\eta_{ji} + \tau_j)
    = (\eta_{\ell k} + \tau_\ell) - (\eta_{jk} + \tau_j).
  \end{equation}
  Now, thanks to \eqref{eqn:flat1}, the left side of
  \eqref{eqn:flat2} lies in $s_i \wedge V$ and so is
  decomposable.  The right side therefore satisfies the
  Pl\"ucker relation and we deduce that
  $s_\ell \wedge s_k' \wedge s_j \wedge s_k'' = \{0\}$ so
  that $s_\ell, s_k', s_j, s_k''$ span a $3$-plane.
  However, this $3$-plane contains
  $f_k = s_k' \oplus s_k'' = s_{jk} \oplus s_{\ell k}$, and
  so $f_j = s_j \oplus s_{jk}$ and
  $f_\ell = s_\ell \oplus s_{\ell k}$.  Thus, our $3$-plane
  contains all four intersections $s_{ij}$, contradicting
  the second order regularity of $f$.
\end{proof}

\subsubsection{Applicability via \KM\ pair of \K\ maps}

Since the connections $\gamma^b,\gamma^w$ are flat, we may
choose a $\gamma^b$-parallel section $x^b$ of $X^b$ and a
$\gamma^w$-parallel section $x^w$ of $X^w$ and, using these,
define $s < f$ by
\begin{subequations}\label{eq:2}
  \begin{equation}\label{eq:8}
    s\restr{\dom_b} = x^b\restr{\dom_b}, \quad 
    s\restr{\dom_w} = x^w\restr{\dom_w}.
  \end{equation}
  $s$ is uniquely determined by its values at a pair of
  initial vertices, one black and one white.  By choosing
  the initial values away from a countable set, we assume
  that $s$ never coincides with an intersection $s_{ij}$.
  We then define $\tau \in \Gamma(\Wedge^2 f)$ by
  \begin{equation}\label{eq:9}
    [\tau, 1]\restr{\dom_b} = x^w\restr{\dom_b},
    [\tau, 1]\restr{\dom_w} = x^b\restr{\dom_w},
  \end{equation}
\end{subequations}
and our assumption on the intersections $s_{ij}$ ensures
that $\tau$ is never $\infty$ thanks to
\cref{lem:gijisomorphism}.  With $s,\tau$ so defined, we
have, on each edge $ij$, $s_i=g_{ij}([\tau_j,1])$ so that
\begin{equation}
  \label{eq:10}
  \eta_{ji}+\tau_j\in s_i\wedge V.
\end{equation}

We are about to prove that $(s,\eta+\d\tau)$ is a \K\ net
but first we show that it satisfies the regularity
conditions of \cref{th:19}.  
\begin{lemma}\label{th:3}
  Suppose that $s < f$ with $s_i \wedge s_j \neq 0$,
  equivalently $s_i\neq s_{ij}$,  on each
  edge $ij$.  Then, for any quadrilateral $\ijkl$,
  $s_i,s_j,s_k,s_{\ell}$ are pairwise distinct while
  $\dim (s_i + s_j + s_k + s_\ell) \geq 3$.
\end{lemma}
\begin{proof}
  Since $s_i \wedge s_j \neq \{0\}$, with
  $U := s_i + s_j + s_k + s_\ell$, we have $\dim U \geq 2$.
  Suppose now that $\dim U = 2$.  Then
  $U = s_i \oplus s_j < f_{ij}$, and similarly
  $U < f_{k\ell}$ so that
  $U = f_{ij} \cap f_{k\ell} = s_{i\ell} \oplus s_{kj}$.  In
  the same way,
  $U = f_{i\ell} \cap f_{jk} = s_{ij} \oplus s_{k\ell}$.
  But the second order regularity gives
  $(s_{ij} \oplus s_{k\ell})\cap (s_{i\ell} \oplus s_{kj}) =
  \{0\}$, and so a contradiction.

  Second order regularity also gives $f_i\cap
  f_k=\set0=f_j\cap f_{\ell}$ so diagonal vertices are also
  pairwise distinct.
\end{proof}

With this in hand, we have:
\begin{proposition}
  Let $(f,\eta)$ be applicable and define $s, \tau$ by
  \eqref{eq:2}.  Then $(s,\eta+\d\tau)$ is a \K\ net.
\end{proposition}
\begin{proof}
  The only thing to prove is that
  $(\eta+\d\tau)_{ji}\in s_i\wedge s_j$ on each edge $ij$.
  Clearly $\tau_i\in s_i\wedge V$ so that, by \eqref{eq:10},
  \begin{equation*}
    (\eta+\d\tau)_{ji}=(\eta_{ji}+\tau_j)-\tau_i\in
    s_i\wedge V.
  \end{equation*}
  By the same argument,
  $(\eta+\d\tau)_{ji}=-(\eta+\d\tau)_{ij}\in s_j\wedge
  V$ so that
  $(\eta+\d\tau)_{ji}\in (s_i\wedge V)\cap(s_j\wedge
  V)=s_i\wedge s_j$ as required.
\end{proof}

We can say more: choose distinct initial conditions for two
parallel sections each of $X^b, X^w$ to arrive at pointwise
distinct maps $s^\pm < f$ and $\tau^\pm$ with
$\eta^\pm := \eta + \d \tau^\pm$ satisfying
\[
  \eta^\pm_{ij} \in s^\pm_j \wedge s^\pm_i
\]
on all edges.  Since $s^\pm$ are pointwise distinct, we have
$f = s^+ \oplus s^-$.  Finally, set $\tau = \tau^- - \tau^+$
so that
\[
  \eta^- = \eta^+ + \d \tau.
\]
Observe that $\tau$ is never zero: indeed if $\tau_j = 0$,
and $ij$ is an edge, then $\tau^{\pm}_j$ coincide so that
$s^{\pm}_i=g_{ij}([\tau_j^{\pm},1])$ coincide also.  In view
of \cref{th:2}, we have therefore arrived at the following
characterisation of an applicable line congruence.

\begin{theorem}\label{th:4}
  Let $f:\dom\to G_2(V)$.  Then the following are
  equivalent:
  \begin{compactenum}
  \item $(f, \eta)$ is a regular applicable line
    congruence.
  \item $f$ is spanned by a \KM\ pair of \K\ nets
    $s^\pm < f$ and regular.
  \end{compactenum}
  In this case $[\eta]=[\eta^+]=[\eta^-]$.
\end{theorem}
\begin{proof}
  The only thing left to prove is that the span of a \KM\
  pair is a line congruence.  However, the Moutard equation
  \eqref{eqn:Moutard} on vertical quadrilaterals assures us
  that $s^{\pm}_i$ and $s^{\pm}_j$ are coplanar so that
  $f_i$ and $f_j$ intersect.
\end{proof}
\begin{remark}\label{th:53}
  Of course, the \K\ nets of \cref{th:4} are far from
  unique: their values may be chosen freely on a pair of
  initial vertices, one black and one white,
  c.f.~\cite[Lemma~3.3]{burstall_discrete_2018}.
\end{remark}
\subsubsection{Applicability via \K\ dual lifts}

Let $f:\dom\to G_2(V)$ be a line congruence and suppose it
is spanned by sections $\sigma^{\pm}$ which are \K\ dual:
thus $\sigma^{\pm}:\dom\to V$ with
\begin{subequations}
  \label{eq:11}
  \begin{align}
    \label{eq:12}
    \d\sigma^+_{ij}\wedge\d\sigma^-_{ij}&=0,\quad\text{on
                                              each edge
                                              $ij$;}\\
    \label{eq:13}
    \d\sigma^+\curlywedge\d\sigma^-&=0,\quad\text{or,
                                         equivalently, $A(\sigma^+,\sigma^-)=0$.}
  \end{align}
\end{subequations}
We emphasise that here we do not require that $\sigma^{\pm}$
take values in an affine subspace of $V$ or even that they
have planar quadrilaterals.

In this situation, set $s^{\pm}=\spn{\sigma^{\pm}}$ and then
define $\eta^{\pm},\tau$ by
\begin{equation}\label{eq:15}
  \eta^{\pm}=\d\sigma^{\mp}\curlywedge\sigma^{\pm}\qquad
  \tau=\sigma^+\curlywedge\sigma^-.
\end{equation}
Then each $\eta_{ji}^{\pm}\in s_i^{\pm}\wedge s_j^{\pm}$ by
\eqref{eq:12} while $\d\eta^{\pm}=0$ by \eqref{eq:13}.
Finally $\eta^-=\eta^{+}+\d\tau$ so that $s^{\pm}$ are a
\KM\ pair and $f$ is applicable.

If $\mu^{\pm}$ are the corresponding Moutard lifts then
\cref{cor:K-Mconsistency} tells us that
$\tau=\mu^-\curlywedge\mu^+$ so that there is a function
$r:\dom\to\R^{\times}$ with
\begin{equation}\label{eqn:konigsToMoutard}
  \mu^\pm = \mp r^{\pm 1} \sigma^{\pm}.
\end{equation}
With $\d\sigma_{ji}^-=\lambda_{ij}\d\sigma^+_{ji}$, we have
\begin{equation*}
  \begin{split}
    r_ir_j\sigma^+_j\wedge\sigma^+_i=
    \mu^+_j\wedge\mu^+_i=\eta^+_{ji}=\half\d\sigma^-_{ji}\wedge(\sigma^+_i+\sigma^+_j)
    \\=\half\lambda_{ij}\d\sigma^+_{ji}\wedge(\sigma^+_i+\sigma^+_j)=
    \lambda_{ij}\sigma^+_j\wedge\sigma^+_i
  \end{split}
\end{equation*}
and so we arrive at the \emph{Christoffel formula}:
\begin{equation}\label{eqn:Christoffel}
  \d \sigma^-_{ij} = r_i r_j \d \sigma^+_{ij}.
\end{equation}
In terms of $\mu^{\pm}$, this last reads
\begin{equation}\label{eqn:moutardToKonigs}
  \mu^-_j - \mu^+_i = \frac{r_i}{r_j}(\mu^-_i - \mu^+_j)
\end{equation}
which is a refinement of \eqref{eqn:MoutardOnTheSides}.

Conversely, if \eqref{eqn:moutardToKonigs} holds, we can
reverse this argument and starting from $\mu^{\pm}$, obtain
$\sigma^{\pm}$ from \eqref{eqn:konigsToMoutard} for which
\eqref{eq:15} holds so that $\sigma^{\pm}$ are \K\ dual.

In fact, \eqref{eqn:moutardToKonigs} always holds and we have:
\begin{proposition}\label{prop:konigsOmega}
  Let $f : \dom \to G_2(V)$ be a line congruence.  Then $f$
  is spanned by \K\ dual sections if and only if $f$ is an
  applicable net.
\end{proposition}
\begin{proof}
  Let $f=s^+\oplus s^-$ be applicable, and let
  $\mu^\pm \in \Gamma s^\pm$ be Moutard lifts of $f$ with
  \eqref{eqn:MoutardOnTheSides}, that is,
  \[
    \mu^-_j - \mu^+_i = b_{ji}(\mu^-_i - \mu^+_j)
  \]
  for some discrete function $b_{ij}$ defined on oriented
  edges $ij$ such that $b_{ij} = 1/b_{ji}$.  From the
  discussion above, we only need to show that
  there is some function $r$ such that $b_{ji} = r_i/r_j$,
  or, equivalently, that
  \[
    b_{i\ell}b_{\ell k}b_{kj}b_{ji} = 1
  \]
  on any quadrilateral $\ijkl$.
	
  For this last, consider the following Moutard cube in the
  sense of \cite{bobenko_discrete_2008}:
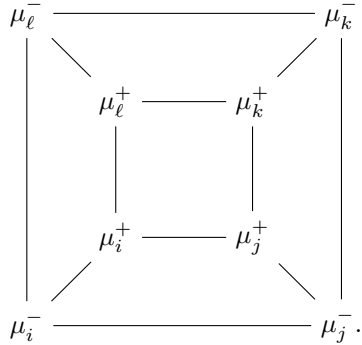
  \[
    \begin{tikzpicture}[scale=0.9]
      \node (A) at (-1,-1) {$\mu^+_i$};
      \node (B) at (1,-1) {$\mu^+_j$};
      \node (C) at (1,1) {$\mu^+_k$};
      \node (D) at (-1,1) {$\mu^+_\ell$};
      \node (E) at (-2.3,-2.3) {$\mu^-_i$};
      \node (F) at (2.3,-2.3) {$\mu^-_j.$};
      \node (G) at (2.3,2.3) {$\mu^-_k$};
      \node (H) at (-2.3,2.3) {$\mu^-_\ell$};
      \path
      (A) edge (B)
      (B) edge (C)
      (C) edge (D)
      (D) edge (A)
      (E) edge (F)
      (F) edge (G)
      (G) edge (H)
      (H) edge (E)
      (A) edge (E)
      (B) edge (F)
      (C) edge (G)
      (D) edge (H);
    \end{tikzpicture}
  \]
  There is a Moutard equation (and so a $b$) relating the
  four vertices of each face of the cube.  The content of
  equation (2.52) of \cite[Theorem
  2.34]{bobenko_discrete_2008} is that suitable ratios of
  the $b$'s from any pair of opposite faces coincide.  In
  our case, this yields
  \[
    \frac{b_{k\ell}}{b_{ji}} = \frac{b_{kj}}{b_{\ell i}}
  \]
  and so the desired conclusion.
\end{proof}

\begin{remark}
  \cref{prop:konigsOmega} guarantees the existence of
  applicable congruences through any \K\ net $s$: simply
  take the line spanned by an affine lift $F=\sigma^+$ of $s$
  and the \K\ dual affine lift $\dual{F}=\sigma^-$ provided by
  \cref{th:1}.
\end{remark}

We conclude this section with a summary of our discussion:
\begin{theorem}\label{thm:applicability}
  Let $f$ be a regular line congruence.  Then the
  following are equivalent:
  \begin{itemize}
  \item $f$ is applicable;
  \item $f$ is spanned by a \KM\ pair of \K\ nets;
  \item $f$ is spanned by Moutard sections satisfying
    \eqref{eqn:MoutardOnTheSides};
  \item $f$ is spanned by \K\ dual sections.
  \end{itemize}
\end{theorem}


\section{Isothermic nets}
\label{sec:isothermic-nets}
We now restrict attention to nets and line congruences
taking values in a non-singular quadric.  The quadric
reduces the ambient projective geometry to conformal
geometry of some signature $(p,q)$.  In our application to
$\Omega$-nets, $(p,q)=(3,1)$ so we will emphasise the case of
indefinite signature.

So contemplate the pseudo-Euclidean space
$\mathbb{R}^{p+1,q+1}$, a $(p+q+2)$-dimensional space
equipped with a non-degenerate symmetric bilinear form
$(\cdot, \cdot)$ of signature $(p+1,q+1)$.

Let
$\mathcal{L} := \{ x \in \mathbb{R}^{p+1,q+1}\st (x,x) = 0\}$
be the light cone, and
$\QQ = \QQ^{p,q}:=
\P(\mathcal{L})=\set{\spn{x}\in\P(\R^{p+1,q+1})\st
  x \in\cL\setminus\set0} \subset
\mathbb{P}(\mathbb{R}^{p+1,q+1})$ be the projective light
cone.  Thus $\QQ$ is a non-singular quadric.
It carries an $\Ortho(p+1,q+1)$-invariant conformal structure
of signature $(p,q)$.

We identify the Lie algebra $\mathfrak{o}(p+1,q+1)$ with the
exterior algebra $\Wedge^2\mathbb{R}^{p+1,q+1}$ via
\begin{equation}\label{eq:14}
  x \wedge y (z) = (x,z) y - (y,z) x
\end{equation}
for $x,y,z \in \mathbb{R}^{p+1,q+1}$.

On maps $s:\dom\to\QQ\subset\P(\R^{p+1,q+1})$ we impose 
regularity assumptions extending those of \cref{th:19}:
\begin{assumption}
  \label{th:34}
  On each quadrilateral $\ijkl$:
  \begin{compactenum}
  \item $s_i,s_j,s_k,s_{\ell}$ are pair-wise distinct;
  \item $s_i,s_j,s_k,s_{\ell}$ are not collinear;
  \item diagonals are non-isotropic.  That is, the
    projective lines $s_is_k$ and $s_js_{\ell}$ do not lie
    in $\QQ$.
  \end{compactenum}
\end{assumption}

\subsection{Isothermic nets and the Moutard equation}
\label{sec:isoth-nets-mout}

\begin{definition}[Isothermic net]
  A map $s:\dom\to\QQ\subset\P(\R^{p+1,q+1})$ is
  \emph{isothermic} if it is \K\ as a map into
  $\P(\R^{p+1,q+1})$.
\end{definition}
Thus $s$ is isothermic if and only if there is a closed,
never-zero $1$-form $\eta$ with
$\eta_{ij}\in s_i\wedge s_j\leq\Wedge^2\R^{p+1,q+1}$, for
each edge $ij$.  

According to \cref{thm:isothermicMoutard}, $s$ is isothermic
exactly when it admits a Moutard lift $\mu \in \Gamma s$ and
then $\eta_{ji} = \mu_j \wedge \mu_i$.  However, in our
conformal setting, the Moutard equation \eqref{eqn:Moutard}
becomes much more rigid and reads:
\begin{equation}
  \label{eq:38}
  \mu_k-\mu_i=
  \frac{(\mu_i,\mu_{\ell}-\mu_j)}{(\mu_{\ell},\mu_j)}(\mu_{\ell}-\mu_j).
\end{equation}
Indeed, \eqref{eqn:Moutard} tells us that
\begin{equation}\label{eq:66}
  \mu_k=\mu_i+c(\mu_{\ell}-\mu_j),
\end{equation}
for some $c\in\R^{\times}$.  However, taking the inner
product of \eqref{eq:66} with itself and using the vanishing
of $(\mu,\mu)$ repeatedly allows us to solve for $c$ and
arrive at \eqref{eq:38}.

Moreover, $\mu$ gives rise to an edge-labelling
$m_{ij} := \frac{1}{(\mu_i, \mu_j)}\in\R\cup\set{\infty}$,
that is \cite[Theorem~4.5]{bobenko_discrete_2008},
\begin{equation}\label{eq:43}
  m_{ij} = m_{k\ell}, \qquad m_{i\ell} = m_{jk}.
\end{equation}
In addition, our regularity assumptions assure us that
$m_{ij}\neq m_{i\ell}$.

When $m$ is finite (which is guaranteed by regularity of $s$
if $q=0$) \eqref{eq:38} tells us $m$ is a cross ratio
factorising function
\cite[Lemma~3.5]{burstall_discrete_2018}, that is, on any
quadrilateral $\ijkl$, the vertices $s_i,s_j,s_k,s_{\ell}$
lie on a nonsingular conic with cross ratio
\begin{equation}\label{eq:17}
  [s_i, s_j, s_k, s_\ell] = \frac{m_{jk}}{m_{ij}}.
\end{equation}
\begin{remark}
  \label{th:42}
  On an edge $ij$ where $m_{ij}$ is finite, $m_{ij}$ is
  equivalent data to $\eta_{ij}$ since $s_i\wedge s_j$ is
  $1$-dimensional and $m_{ij}$ fixes the scale.  In
  particular, if we scale $\eta$, then $m_{ij}$ scales
  reciprocally.
\end{remark}

\subsection{\KM\ transformations of isothermic
  nets}
\label{sec:darb-km-transf}

Let $(s^{\pm},\eta^{\pm}):\dom\to\QQ$ be a \KM\ pair of
isothermic nets with Moutard lifts $\mu^{\pm}$.  Thus there
is an isothermic net
$s=s^+\sqcup s^{-}:\set{0,1}\times\dom\to\QQ$
\begin{equation*}
  s\restr{\set{0}\times\dom}=s^+\qquad
  s\restr{\set{1}\times\dom}=s^-
\end{equation*}
and Moutard lift $\mu=\mu^{+}\sqcup\mu^{-}$.  In particular,
the edge-labelling property for $\mu$ on vertical edges says
that $(\mu_{(0,i)},\mu_{(1,i)})=(\mu_{(0,j)},\mu_{(1,j)})$
so that $m:=1/(\mu^+,\mu^-)$ is constant.  When $m$ is
finite, $s^{\pm}_{i},s^{\pm}_{j}$ lie on a non-singular
conic while \eqref{eq:17} on vertical faces reads
\begin{equation*}
  [s_{i}^+,s_{j}^+,s_j^-,s_i^-]=\frac{m}{m_{ij}}
\end{equation*}
so that we recognise that $s^-$ is precisely a Darboux transform
\cite[Definition~4.7]{bobenko_discrete_2008} of $s^+$ with
parameter $m$.  In view of this, we mildly extend the notion
of Darboux transform to include the case $m=\infty$:
\begin{definition}[Darboux transform, Darboux pair]
  Let $(s^{+},\eta^{+}):\dom\to\QQ$ be isothermic.  A \KM\ transform
  $(s^-,\eta^{-})$ with $1/(\mu^+,\mu^-)\equiv m\in\R\cup\set{\infty}$ is
  called a \emph{Darboux transform of $s^+$ with parameter
    $m$} and $s^+,s^-$ are called a \emph{Darboux pair with
    parameter $m$}.

  An \emph{isotropic Darboux pair} is a Darboux pair with
  parameter $m=\infty$.
\end{definition}

The nets $s^{\pm}$ of an isotropic Darboux pair are
orthogonal and so span an applicable congruence of lines
lying in $\QQ$ by \cref{th:4}.  We shall have more to say
about such congruences below in
Section~\ref{sec:appl-legendre-maps}.

\subsection{Duality for isothermic nets}
\label{sec:dual-isoth-nets}

\subsubsection{Circular nets in $\R^{p,q}$}
\label{sec:circular-nets-rp}

Let $\fo,\q\in\cL\subset\R^{p+1,q+1}$ with $(\fo,\q)=-1$.
Set $\R^{p,q}:=\spn{\fo,\q}^{\perp}$ and let
\begin{equation*}
  E=\set{v\in\cL\st (v,\q)=-1}.
\end{equation*}
Then $\phi:\R^{p,q}\to E$ given by
\begin{equation}
  \label{eq:16}
  \phi(x):=\fo + x + \half(x,x)\q
\end{equation}
is an isometry with inverse $\psi:=\pi\restr{E}$ for
$\pi:\R^{p+1,q+1}\to\R^{p,q}$ orthoprojection. Meanwhile,
the projection $\cL\to\P(\cL)$ restricts to a conformal
diffeomorphism
$E\cong\P(\cL)\setminus\P(\cL\cap{\q}^{\perp})$.
Putting these together yields \emph{stereoprojection}
$\P(\cL)\setminus\P(\cL\cap\q^{\perp})\cong\R^{p,q}$
with inverse $x\mapsto\spn{\phi(x)}$.

We note that, for $x_1,x_2\in\R^{p,q}$ and $y_i=\phi(x_i)$, we have
\begin{equation}
  \label{eq:24}
  (y_1,y_2)=-\half(x_1-x_2,x_1-x_2).
\end{equation}

\begin{definition}[Euclidean lift]
  Let $s:\dom\to\QQ$ with stereoprojection
  $x:\dom\to\R^{p,q}$.  We call
  $y:=\phi(x)=\fo+x+\half(x,x)\q\in\Gamma s$ the
  \emph{Euclidean lift} of $s$ (or $x$) with respect to
  $\q$.

  It is the unique section $y$ of $s$ with $(y,\q)\equiv-1$.
\end{definition}

A circular net in $\R^{p,q}$ is the stereoprojection of a
$Q$-net in $\QQ\setminus\P(\cL\cap\q^{\perp})$:
\begin{definition}[Circular net]
  $x: \dom \to \mathbb{R}^{p,q}$ is called a \emph{circular
    net} if $x$ has non-collinear quadrilaterals and, on
  each such quadrilateral $\ijkl$, its inverse
  stereoprojection $s:\dom\to\QQ$ has
  \begin{equation*}
    \dim s_{\ijkl}=3,
  \end{equation*}
  where $s_{\ijkl}=s_i+s_j+s_k+s_{\ell}$.
\end{definition}

\begin{remarks}\label{th:5}
\item[]
  \begin{compactenum}
  \item It is easy to see that a circular
    net has (affine) planar quadrilaterals and so is a
    $Q$-net in $\R^{p,q}$.
  \item A \emph{circle} in $\R^{p,q}$ is the intersection of
    an affine $2$-plane with a quadric cone of the form
    $\set{x\in\R^{p,q}\st (x-c,x-c)=R}$, for some $c\in\R^{p,q}$ and
    $R\in\R$.  We note that $x\in\R^{p,q}$ lies on such a
    circle if and only if its Euclidean lift
    $\fo+x+\half(x,x)\q$ is orthogonal to
    $\fo+c+\half((c,c)-R)\q$ in $\R^{p+1,q+1}$.

    In indefinite signature, circles need not be
    $1$-dimensional: null affine $2$-planes are circles.

    One can show that $x$ is circular exactly when the
    vertices of each quadrilateral lie on a
    circle: an element of
    $s_{\ijkl}^{\perp}\setminus\q^{\perp}$ is, up to scale,
    of the form $\fo+c+\half((c,c)-R)\q$.
  \end{compactenum}
\end{remarks}

For later use, we record:
\begin{lemma}
  \label{th:6}
  Let $x:\dom\to\R^{p,q}$ be a circular net with inverse
  stereoprojection $s$ and let $\ijkl$ be a
  quadrilateral.  Set $U_{\ijkl}:=\spn{\d x_{ij}, \d
    x_{jk}, \d x_{k\ell}, \d x_{\ell i}}$ and
  $W_{\ijkl}:=s_{\ijkl}\cap\q^{\perp}$.

  Then $\dim U_{\ijkl}=\dim W_{\ijkl}=2$ and orthoprojection
  $\pi:\R^{p+1,q+1}\to\R^{p,q}$ restricts to an isomorphism
  $W_{\ijkl}\cong U_{\ijkl}$.

  In particular,
  $\Wedge^2 \pi\restr{\Wedge^2 W_{\ijkl}} : \Wedge^2 W_{\ijkl}
  \to \Wedge^2 U_{\ijkl},\Wedge^2 \pi (a \wedge b) = \pi(a)
  \wedge \pi(b)$ is also an isomorphism.
\end{lemma}
\begin{proof}
  First note that $s_{\ijkl}$ is not contained in
  $\q^{\perp}$ since no $s_{i}$ lies in $\q^{\perp}$.  Thus
  $\dim W_{\ijkl}=2$.  
  Moreover, $W_{\ijkl}=\spn{\d y_{ij},\d y_{jk},\d y_{k\ell},\d
    y_{\ell i}}$ since $s_{\ijkl}=\spn{y_i,y_j,y_k,y_{\ell}}$, 
   while $\pi(\d y)=\d x$ so that
  $\pi(W_{\ijkl})=U_{\ijkl}$.  Now the quadrilateral is
  non-collinear so that $\dim U_{\ijkl}=2$ also, whence
  $\pi\restr{W_{\ijkl}}$ is an isomorphism.
\end{proof}
\begin{remark}
  Note that
  \begin{equation*}
    W_{\ijkl}\cap\ker\pi=s_{\ijkl}\cap\q^{\perp}\cap\spn{\fo,\q}
    =s_{\ijkl}\cap\spn{\q}
  \end{equation*}
  so that a quadrilateral $\ijkl$ of $x$ is collinear
  exactly when $\q\in s_{\ijkl}$.  Thus the stereoprojection
  of a $Q$-net $s$ with respect to $\fo,\q$ is circular in
  our sense so long as $\q\notin s_{\ijkl}$ for any
  elementary quadrilateral $\ijkl$.  This amounts to
  choosing $\q$ off a set of measure zero in $\QQ$.
\end{remark}

We conclude our present discussion of circular nets with a result
which is ``obvious'' \cite[page~156]{bobenko_discrete_2008}
in the definite case but less so (at least to us!) in the
present setting:
\begin{proposition}
  \label{th:9}
  Let $x:\dom\to\R^{p,q}$ be a circular net and
  $\dual{x}:\dom\to\R^{p,q}$ an edge-parallel net.  Then
  $\dual{x}$ is also circular.
\end{proposition}
\begin{proof}
  We work on a single quadrilateral $\ijkl$.  By translation
  and scaling, we may assume without loss of generality that
  $x_i = \dual{x}_i$ and $x_j = \dual{x}_j$.  Also without loss
  of generality, assume that $\d x_{i\ell},\d x_{k\ell}$
  span $U_{\ijkl}$ and write
  \begin{align*}
    \d x_{ij} &= \alpha \d x_{i\ell} + \beta \d
                x_{k\ell} \\
    \d x_{jk} &= \gamma \d x_{i\ell} + \delta \d
                x_{k\ell},
  \end{align*}
  for $\alpha,\beta,\gamma,\delta\in\R$.  Then, for some
  $t\in\R$,
  $\d \dual{x}_{jk} = t \d x_{jk} = t (\gamma \d x_{i\ell} +
  \delta \d x_{k\ell})$ from which we deduce, using
  $\d\dual{x}_{k\ell}\prl\d x_{k\ell}$ that
  \begin{subequations}\label{eq:18}
    \begin{align}
      \d\dual{x}_{i \ell}&=(\alpha + \gamma t)\d
                           x_{i\ell}\\
      \d\dual{x}_{k \ell}&=-(\beta+\delta t)\d x_{k \ell}.
    \end{align}
  \end{subequations}

  Now let $y,\dual{y}$ be the Euclidean lifts of
  $x,\dual{x}$ and contemplate
  \[
    p(t) := \dual{y}_i \wedge \dual{y}_j \wedge \dual{y}_k
    \wedge \dual{y}_\ell,
  \]
  a polynomial with values in $\Wedge^4\R^{p+1,q+1}$ which
  vanishes exactly when $\dual{x}$ is circular.

  In view of \eqref{eq:18}, $p(t)$ is cubic in $t$ if
  $\gamma\delta\neq 0$ and quadratic otherwise.  However,
  $p(t)$ has roots at $0$, when $\dual{x}_k = \dual{x}_j$;
  at $1$, when $x=\dual{x}$; at $-\alpha/\gamma$, when
  $\dual{x}_i=\dual{x}_{\ell}$, if $\gamma\neq 0$ and at
  $-\beta/\delta$, when $\dual{x}_k=\dual{x}_{\ell}$, if
  $\delta\neq 0$.  In any case, $p(t)$ has four roots when
  it is cubic and three when quadratic and so vanishes
  identically.  Thus any edge-parallel $\dual{x}$ is
  circular.
\end{proof}

\subsubsection{The Christoffel dual}
\label{sec:christoffel-dual}

With these preparations in hand, we show that, just as in
the definite case
\cite[Theorem~4.32]{bobenko_discrete_2008}, a net in $\QQ$
is isothermic if and only if its stereoprojection is a
circular \K\ net.  More precisely,
\begin{theorem}
  \label{th:8}
  Let $(s,\eta):\dom\to\QQ\setminus\P(\cL\cap\q^{\perp})$ be a
  $Q$-net with (circular) stereoprojection $x$.  Then $s$ is
  isothermic if and only if $x$ has a \K\ dual $\dual{x}$:
  that is, $\dual{x}$ is edge-parallel to $x$ with
  \begin{equation}\label{eq:19}
    \d x\curlywedge\d\dual{x}=0.
  \end{equation}
  In this case, $\dual{x}$ is also the stereoprojection of
  an isothermic net.

  In view of this, we say that $x:\dom\to\R^{p,q}$ is
  isothermic if it is a circular \K\ net.  We call
  $\dual{x}$ a \emph{Christoffel dual} of $x$.  It is
  determined up to translation and a constant scaling (and
  then $\eta$ is scaled in the same way).
\end{theorem}
\begin{proof}
  First suppose that $s$ is isothermic so that the Euclidean
  lift $y$ of $s$ takes values in the affine space
  $\A=\set{v\st (\q,v)=-1}$ and let $\dual{y}$ be the \K\
  dual of $y$ as in \cref{th:1}.  Thus
  $\eta=\d\dual{y}\curlywedge y$.  Set
  $\dual{x}=\pi\dual{y}$.  Then $\dual{x}$ is edge-parallel
  to $x$ since $\dual{y}$ is edge-parallel to $y$.
  Furthermore, we have $\d y\curlywedge\d\dual{y}=0$ and taking
  $\Wedge^2\pi$ of this yields \eqref{eq:19}.

  For the converse, we must exploit the circularity of $x$.
  Define a $\q^{\perp}$-valued $1$-form $\omega$ by
  \begin{equation*}
    \omega=\d\dual{x}+(x\wedge\d\dual{x})\q
  \end{equation*}
  and observe that each $\omega_{ji}\prl \d y_{ji}=\d
  x_{ji}+(x\wedge\d x)_{ji}\q$ since $\d\dual{x}_{ji}\prl \d x_{ji}$.
  We now define $\eta:=\omega\curlywedge y$ and deduce that each
  $\eta_{ji}\prl y_{i}\wedge y_j$ and so takes values in
  $s_i\wedge s_j$.  Moreover, on a quadrilateral $\ijkl$,
  both $\omega$ and $\d y$ take values in $W_{\ijkl}$ while
  \begin{align*}
    \pi(\d\omega)&=\d^{2}\dual{x}=0\\
    \Wedge^2\pi(\omega\curlywedge\d y)&=\d\dual{x}\curlywedge\d x=0.
  \end{align*}
  Thus \cref{th:6} tells us that $\d\omega$, $\omega\curlywedge\d
  y$ and so $\d\eta$ vanish.  Thus $s$ is isothermic.

  Finally, thanks to \cref{th:9}, $\dual{x}$ is also
  circular with \K\ dual $x$ and so is also the
  stereoprojection of an isothermic net.
\end{proof}

\begin{remark}
  \label{th:10}
  Note that we may express $\d\dual{x}$ directly in terms of
  $\eta$: the argument of \cref{th:1} tells us that
  $\d\dual{y}=\eta\q$ so that:
  \begin{subequations}
    \label{eq:20}
    \begin{align}
      \label{eq:21}
      \eta&=(\eta\q)\curlywedge y\\
      \label{eq:22}
      \d\dual{x}&=\pi(\eta\q).
    \end{align}
  \end{subequations}
\end{remark}

The scaling between $\d x$ and $\d\dual{x}$ can be expressed
directly in terms of the edge labelling $m_{ij}$:
\begin{proposition}
  \label{th:11}
  Let $x:\dom\to\R^{p,q}$ be isothermic with edge-labelling
  $m_{ij}$ and Christoffel dual $\dual{x}$.  Then, for each
  edge $ij$,
  \begin{equation}
    \label{eq:23}
    (\d x_{ij},\d\dual{x}_{ij})=-\frac{2}{m_{ij}}.
  \end{equation}
\end{proposition}
\begin{proof}
  Let $s$ be the inverse stereoprojection of $x$ and
  $y,\mu\in\Gamma s$ the Euclidean and Moutard lifts.  We
  have $\mu=ry$ where $r=-(\mu,\q)$.  Now \eqref{eq:24}
  gives
  \begin{equation}\label{eq:25}
    \frac{1}{m_{ij}}=(\mu_i,\mu_j)=r_ir_j(y_i,y_j)=-r_ir_j\half(\d
    x_{ij},\d x_{ij}),
  \end{equation}
  while, on the other hand,
  \begin{equation*}
    \eta_{ij}\q=(\mu_i\wedge\mu_j)\q=r_ir_j\d y_{ij}
  \end{equation*}
  so that \eqref{eq:22} gives
  \begin{equation}\label{eq:26}
    \d\dual{x}_{ij}=r_ir_j\d x_{ij}.
  \end{equation}
  Putting \eqref{eq:25} and \eqref{eq:26} together yields
  \eqref{eq:23}.
\end{proof}

Since $x$ is the Christoffel dual of $\dual{x}$, we
immediately learn:
\begin{corollary}
  \label{th:12}
  An isothermic net and its Christoffel dual have the same
  edge-labelling.
\end{corollary}

\cref{th:11} can be interpreted in two interesting ways.
When $\Sigma=\Z^n$ and all $m_{ij}$ are finite, we recover
the well-known Christoffel formula \cite[\S5.7.7]{Her03}:
\begin{equation*}
  \d\dual{x}_{ij}=-\tfrac{2}{m_{ij}}\d x_{ij}/(\d x_{ij},\d x_{ij}).
\end{equation*}
On the other hand, if $\hat{x}$ is a Darboux transform of
$x$ on $\Z^n$, we may apply \cref{th:8} and \cref{th:11} to
$x\sqcup \hat{x}$ on $\dom=\set{0,1}\times\Z^n$ to obtain a
result due, in the classical smooth case, to Bianchi
\cite[p.~105]{bianchi_ricerche_1905} (see
\cite[\S5.7.32]{Her03} for the discrete definite case):
\begin{corollary}
  \label{th:13}
  Let $x:\Z^n\to\R^{p,q}$ be isothermic and
  $\hat{x}:\Z^n\to\R^{p,q}$ a Darboux transform  with
  parameter $m\in\R^{\times}\cup\set{\infty}$.  Let $\dual{x}$
  be a Christoffel dual of $x$.

  Then there is a Christoffel dual $\dual{\hat{x}}$ of
  $\hat{x}$ which is simultaneously a Darboux transform of
  $\dual{x}$ with parameter $m$. Moreover $\hat{x}-x$ and
  $\dual{\hat{x}}-\dual{x}$ are pointwise parallel and
  \begin{equation}
    \label{eq:27}
    (\hat{x}-x,\dual{\hat{x}}-\dual{x})=-\frac{2}{m}.
  \end{equation}
\end{corollary}
\begin{proof}
  Apply \cref{th:8} to $x\sqcup\hat{x}$ and translate to
  get a Christoffel dual
  $(x\sqcup \hat{x})^{\vee}=\dual{x}\sqcup \dual{\hat{x}}$
  of $x\sqcup \hat{x}$ extending $\dual{x}$.  Then
  $\dual{\hat{x}}$ is a Christoffel dual of $\hat{x}$
  and also, thanks to \cref{th:12}, a Darboux transform of
  $\dual{x}$ with parameter $m$ (the edge label for vertical
  edges).  Again, since $(x\sqcup \hat{x})^{\vee}$ is
  edge-parallel to $x\sqcup \hat{x}$ on vertical edges, we
  get that $(\hat{x}-x)\prl(\dual{\hat{x}}-\dual{x})$ while
  \cref{th:11} yields \eqref{eq:27}.
\end{proof}

\subsection{Families of flat connections}
\label{sec:famil-flat-conn}

A defining characteristic of an isothermic net $s$ with
finite cross-ratio factorising function $m$ is a
$1$-parameter family of flat connections
$(\Gamma^s(t))_{t\in\R}$ on the trivial bundle
$\dom\times\R^{p+1,q+1}$ which are defined as follows:
\begin{equation}\label{eq:40}
  \Gamma^s(t)_{ji}  := \Gamma_{s_i}^{s_j} (1 - t/m_{ij}),
\end{equation}
where, for $\lambda\in\R^{\times}$,
\begin{equation*}
  \Gamma_{s_i}^{s_{j}}(\lambda)=
  \begin{cases}
    \lambda&\text{on $s_j$}\\1&\text{on $(s_i\oplus
      s_j)^{\perp}$}\\1/\lambda&\text{on $s_i$}.
  \end{cases}
\end{equation*}

Implicit in the discussion in
\cite[\S3]{burstall_discrete_2018} is an extension of this
to the case where $m_{ij}=\infty$ on one family of edges.
We give an explicit self-contained argument here.
\begin{proposition}
  \label{th:21}
  Let $(s,\eta):\dom\to\QQ$ be isothermic with
  edge-labelling $m$.  Define connections
  $(\Gamma^s(t))_{t\in\R}$ on $\dom\times\R^{p+1,q+1}$ by
  \eqref{eq:40} if $m_{ij}$ is finite and
  \begin{equation}
    \label{eq:41}
    \Gamma^s(t)_{ji}=\exp(t\eta_{ji})
  \end{equation}
  when $m_{ij}=\infty$.

  Then each $\Gamma^s(t)$ is a flat connection: on each
  quadrilateral $\ijkl$,
  \begin{equation}
    \label{eq:42}
    \Gamma^s(t)_{kj}\Gamma^s(t)_{ji}=
    \Gamma^s(t)_{k\ell}\Gamma^s(t)_{\ell i}.
  \end{equation}
\end{proposition}
\begin{proof}
  We suppose that $m_{ij}=m_{k\ell}=\infty$.  \cref{th:34}
  assures us that $s_j,s_{\ell}$ are not orthogonal and that
  $m_{i\ell}$ is finite.  We now follow the
  strategy for the case of finite $m$ in
  \cite[Lemma~4.7]{burstall_isothermic_2011} by proving that
  both sides of \eqref{eq:42} equal
  $\Gamma_{s_j}^{s_{\ell}}(1-t/m_{\ell i})$.  With $L(t)$
  denoting the left side, it is easy to see that $L(t)$ and
  $\Gamma_{s_j}^{s_{\ell}}(1-t/m_{\ell i})$ agree on both
  $s_j$ and $s_j^{\perp}/s_j$ so that, since both lie in
  $\Ortho(p+1,q+1)$, it suffices to show that they agree on
  $s_{\ell}$, that is, with $\mu\in\Gamma s$, the Moutard
  lift,
  \begin{equation*}
    \Gamma^{s_k}_{s_j}(1-t/m_{jk})\exp(t\eta_{ji})\mu_{\ell}=
    (1-t/m_{\ell i})\mu_{\ell},
  \end{equation*}
  or, equivalently,
  \begin{equation*}
    \exp(t\eta_{ji})\mu_{\ell}=
    (1-t/m_{\ell i})\Gamma^{s_j}_{s_k}(1-t/m_{jk})\mu_{\ell}.
  \end{equation*}
  However, a straightforward calculation using \eqref{eq:14}
  and $m_{\ell i}=m_{jk}$ shows that
  this last amounts to \eqref{eq:38}.  The equality for the
  right hand side is similar.
\end{proof}

\begin{remark}\label{th:43}
  The connections $\Gamma^s(t)$ and the $1$-form $\eta$ are
  equivalent data: given $(\Gamma^s(t))_{t\in\R}$ we recover
  $\eta$ by
  \begin{equation*}
    \eta=\partial/\partial t\restr{t=0}\Gamma^s(t).
  \end{equation*}
  Conversely, given $\eta$, we have
  \begin{equation*}
    \Gamma^s_{ji}(t)=\exp\bigl( - m_{ij} \log(1 -
    t/m_{ij}) \eta_{ji} \bigr),
  \end{equation*}
  where we use L'H\^{o}pital's rule to interpret the right side
  when $m_{ij}=\infty$.  In particular, scaling $\eta$ by a
  constant scales the parameter $t$ also.  In more detail,
  replacing $\eta$ by $\lambda\eta$, for
  $\lambda\in\R^{\times}$ a constant, requires us to replace
  $m_{ij}$ by $m_{ij}/\lambda$ and $\Gamma^{s}(t)$ by
  $\Gamma^{s}(\lambda t)$.
\end{remark}

These flat connections give an alternative perspective on
the transformation theory of isothermic nets.  In
particular, we have a discrete analogue of Darboux's linear
system \cite{Dar99e}, see also
\cite{burstall_isothermic_2011,burstall_discrete_2018}:
\begin{proposition}
  \label{th:22}
  Let $s:\dom\to\QQ$ be an isothermic net and
  $m\in\R^{\times}$ not equal to any $m_{ij}$.  Let
  $\hat{s}:\dom\to\QQ$ be pointwise
  non-orthogonal to $s$.

  Then $\hat{s}$ is a Darboux transform of $s$ with
  parameter $m$ if and only if $\hat{s}$ is
  $\Gamma^s(m)$-parallel:
  \begin{equation*}
    \hat{s}_j=\Gamma^s(m)_{ji}\hat{s}_i,
  \end{equation*}
  for all edges $ij$.

  In particular, $\hat{s}$ is uniquely determined by its
  value at a single point of $\dom$.
\end{proposition}
\begin{proof}
  Let $\mu\in\Gamma s$ be the Moutard lift with
  $m_{ij}=1/(\mu_i,\mu_j)$ and let $\hat{\mu}\in\Gamma
  \hat{s}$ be the unique section with
  $(\mu,\hat{\mu})=1/m$.  Then, for any edge $ij$, including
  any with $m_{ij}=\infty$, we have
  \begin{equation}\label{eq:44}
    \Gamma^s(m)_{ji}\hat{\mu}_i=
    \hat{\mu}_{i}-\mu_{j}+\frac{(\mu_j,\hat{\mu}_i)}{(\mu_i,\hat{\mu}_i-\mu_j)}\mu_i.
  \end{equation}
  If $\hat{s}$ is a Darboux transform with parameter $m$,
  then $\hat{\mu}$ is the Moutard lift of $\hat{s}$ and the
  Moutard equation \eqref{eq:38} on the vertical
  quadrilateral tells us that \eqref{eq:44} is a multiple of
  $\hat{\mu}_j$ and so takes values in $\hat{s}_j$.  Thus
  $\hat{s}$ is $\Gamma^s(m)$-parallel.

  Conversely, if $\hat{s}$ is parallel, we have
  \begin{equation*}
    \hat{\mu}_{i}-\mu_{j}+\frac{(\mu_j,\hat{\mu}_i)}{(\mu_i,\hat{\mu}_i-\mu_j)}\mu_i
    =c\hat{\mu}_j,
  \end{equation*}
  for some $c\in\R$.  Taking the inner product with $\mu_j$
  rapidly yields
  \begin{equation*}
    c=\frac{(\mu_j,\hat{\mu}_i)}{(\mu_i,\hat{\mu}_i-\mu_j)}
  \end{equation*}
  so that $\mu,\hat{\mu}$ solve the Moutard equation on
  vertical quadrilaterals.  That $\hat{s}$ is a Darboux
  transform now follows at once from the multidimensional
  consistency of the Moutard equation.
\end{proof}

Again, the flat connections are responsible for the Calapso
transformation and we have the following extension of
\cite[\S2]{burstall_discrete_2014} (see also
\cite[\S5.7.16]{Her03}) to include the case where
$m_{ij}=\infty$:
\begin{proposition}
  \label{th:23}
  Let $(s,\eta):\dom\to\QQ$ be isothermic with Moutard lift
  $\mu$, edge-labelling $m_{ij}$ and flat connections
  $\Gamma^s(t)$.  Let $T(t):\dom\to\Ortho(p+1,q+1)$
  trivialise $\Gamma^s(t)$:
  $\Gamma^s(t)_{ji}=T(t)_j^{-1}T(t)_i$ on each edge
  $ij$.

  Define $s(t):=T(t)s:\dom\to\QQ$.  Then $s(t)$ is
  isothermic with Moutard lift $T(t)\mu$,
  edge-labelling $m(t)_{ij}=m_{ij}-t$ (interpreted as
  $\infty$ if $m_{ij}=\infty$) and flat connections
  given by
  \begin{equation}
    \label{eq:45}
    \Gamma^{s(t)}(u)=T(t)\cdot\Gamma^s(t+u).
  \end{equation}
  We call $s(t)$ a \emph{Calapso transform} of $s$.  It is
  defined up to a constant element $g(t)\in\Ortho(p+1,q+1)$.
\end{proposition}
\begin{proof}
  We start with the flat connections
  $T(t)\cdot\Gamma^s(t+u)$.  A computation using
  $T(t)_i^{-1}T(t)_j=\Gamma^s(t)_{ij}$ reveals that
  \begin{equation}\label{eq:46}
    (T(t)\cdot\Gamma^s(t+u))_{ji}=
    \begin{cases}
      T(t)_{j}\Gamma^{s_j}_{s_i}(1-u/(m_{ij}-t))T(t)_j^{-1}&\text{if
        $m_{ij}\neq\infty$}\\
      T(t)_{j}\exp(u\eta_{ji})T(t)_j^{-1}&\text{if
        $m_{ij}=\infty$}.
    \end{cases}
  \end{equation}
  Define $1$-forms $\eta(t)$ by
  \begin{equation*}
    \eta(t):=\partial/\partial u\restr{u=0}T(t)\cdot\Gamma^s(t+u)
  \end{equation*}
  and use \eqref{eq:46} to get, in all cases,
  \begin{equation*}
    \eta(t)_{ji}=\frac{1}{1-t/m_{ij}}\Ad_{T(t)_j}\eta_{ji}
    =(T(t)\mu)_j\wedge(T(t)\mu)_i,
  \end{equation*}
  where we have used
  \begin{equation}
    \label{eq:47}
    T(t)_j^{-1}T(t)_i\mu_i=\frac1{1-t/m_{ij}}\mu_i.
  \end{equation}
  Now $\eta(t)$ is closed since
  $T(t)\cdot\Gamma^s(t+u)_{ji}$ is flat for all $u$, so
  that $s(t)$ is isothermic with Moutard lift $\mu(t)=T(t)\mu$.
  Moreover, \eqref{eq:47} rapidly yields
  $(\mu(t)_i,\mu(t)_j)=1/(m_{ij}-t)$ which, together with
  \eqref{eq:46}, gives \eqref{eq:45}.
\end{proof}


\section{Applicable Legendre maps}
\label{sec:appl-legendre-maps}

For $p,q\geq 1$, let $\cZ=\cZ^{p,q}$ be the space of
projective lines in $\QQ^{p,q}$ or, equivalently, the
Grassmannian of null $2$-planes in $\R^{p+1,q+1}$.  Then
$\cZ$ is a contact manifold of dimension $2(p+q)-3$. 

\begin{definition}[Legendre map]
  A \emph{Legendre map} is a discrete line congruence
  $f:\dom\to\cZ$.
\end{definition}

We study applicable Legendre maps and their transformations.
The key observation is that, thanks to \cref{th:4} and the
discussion in Section~\ref{sec:darb-km-transf},
$f:\dom\to\cZ$ is an applicable Legendre map if and only if
it is spanned by an isotropic Darboux pair of isothermic
nets.

\subsection{Duality for applicable Legendre maps}
\label{sec:dual-appl-legendre}

With notation as in Section~\ref{sec:dual-isoth-nets}, write
$\R^{p+1,q+1}=\R^{p,q}\oplus\spn{\fo,\q}$ and let $\cZ_{\q}$
denote the set of affine null lines in $\R^{p,q}$.
Inverse stereoprojection identifies $\cZ_{\q}$ with the open
subset of $\cZ$ consisting of lines in $\QQ$ that do not lie
in the quadric at infinity $\P(\cL\cap\q^{\perp})$ (which
only contains lines if $p,q\geq 2$).

If $L:\dom\to\cZ_{\q}$ is the stereoprojection of $f$
then $f$ is Legendre exactly when $L_i,L_j$ are affine
coplanar for each edge $ij$.  We say that $L$ is applicable
if $f$ is.

So let $(f,[\eta])$ be an applicable Legendre map with
stereoprojection $L$.  For any $\eta\in[\eta]$, $\eta\q$ is
closed and so there is $\dual{x}^{\eta}:\dom\to\R^{p,q}$, unique up
to translation, with
\begin{equation}\label{eq:30}
  \d\dual{x}^{\eta}=\pi\eta\q.
\end{equation}
Moreover, for $\tau\in\Gamma\Wedge^2f$, we may take
\begin{equation}\label{eq:28}
  \dual{x}^{\eta+\d\tau}=\dual{x}^{\eta}+\pi\tau\q.
\end{equation}
Let $x_1,x_2$ span $L$ with Euclidean lifts $y_1,y_2$ so
that any $\tau\in\Gamma\Wedge^2f$ is of the form
$\lambda y_2\Wedge y_1$, for some $\lambda:\dom\to\R$.
Then \eqref{eq:28} reads
\begin{equation}
  \label{eq:29}
  \dual{x}^{\eta+\d\tau}=\dual{x}^{\eta}+\lambda(x_2-x_1).
\end{equation}
Thus all $\dual{x}^{\eta+\d\tau}$ lie on the family
$\dual{L}$ of affine null lines through $\dual{x}$ that are
pointwise parallel to $L$. In particular, $\dual{L}$
contains a Christoffel dual of any isothermic net in $L$.
Further, any section of $\dual{L}$ is of the form
$\dual{x}^{\eta+\d\tau}$, for some $\tau\in\Gamma\Wedge^2f$.

We have:
\begin{theorem}
  \label{th:14}
  Let $(f,\eta):\dom\to\cZ$ be an applicable Legendre map
  with stereoprojection $L$, let
  $\dual{x}^{\eta}:\dom\to\R^{p,q}$ solve \eqref{eq:30} and
  let $\dual{L}:\dom\to\cZ_{\q}$ consist of the lines through
  $\dual{x}^{\eta}$ parallel to $L$.

  Then $\dual{L}$ is also the stereoprojection of an
  applicable Legendre map.

  We call $\dual{L}$ the \emph{dual Legendre map to $L$ with
    respect to $\q$}.  It is determined up to translation.
\end{theorem}
\begin{proof}
  $f:\dom\to\cZ$ is spanned by an isotropic Darboux pair
  $(s^{\pm},\eta^{\pm})$ of isothermic nets with
  $\eta^{\pm}\in[\eta]$.  Then
  $\dual{x}^{+}:=\dual{x}^{\eta^+}$ is a Christoffel dual of
  $x^+$ and with $\dual{x}^{-}$ the simultaneous Christoffel
  dual of $x^-$ and Darboux transform of $\dual{x}^{+}$
  provided by \cref{th:13}, we have that $\dual{x}^-$ lies
  on the line through $\dual{x}^+$ in the direction
  $x^+-x^-$ which is $\dual{L}$.  Thus $\dual{L}$ is spanned
  by an isotropic Darboux pair of isothermic nets and so is
  the stereoprojection of an applicable Legendre map as
  required.
\end{proof}

\subsection{Transformations of applicable Legendre maps}
\label{sec:transf-appl-legendre}

Let $f:\dom\to\cZ$ be an applicable Legendre map spanned by
an isotropic Darboux pair $s^{\pm}$.  The key to
\cref{th:14} was to consider the isothermic net
$s=s^+\sqcup s^-:\set{0,1}\times\dom\to\QQ$ and then take
the Christoffel dual of $s$.

The same idea gives us Darboux and Calapso transformations
of applicable Legendre maps:

\subsubsection{Darboux transformations}
\label{sec:darb-transf}

\begin{proposition}\label{th:35}
  Let $f=s^+\oplus s^-:\dom\to\cZ$ be an applicable Legendre
  map spanned by an isotropic Darboux pair of isothermic
  nets and set $s:s^+\sqcup s^-:\set{0,1}\times\dom\to\QQ$.

  Let $m\in\R^{\times}$ be distinct from any $m_{ij}$ of $s$
  and let $\hat{s}=\hat{s}^+\sqcup \hat{s}^-$ be a Darboux
  transform of $s$ with parameter $m$.

  Then $\hat{f}:=\hat{s}^+\oplus\hat{s}^-$ is an applicable
  Legendre map which we call a \emph{Darboux transform of
    $f$ with parameter $m$}.
\end{proposition}
\begin{proof}
  $s$ and $\hat{s}$ have the same edge labelling so that, in
  particular, the label on the vertical edges of $\hat{s}$
  is $\infty$.  It follows at once that $\hat{s}^{\pm}$ are
  an isotropic Darboux pair and so span an applicable
  Legendre map.
\end{proof}

We give another characterisation of the Darboux transform
$\hat{f}$ which replicates the formulation of
\cite[\S2.4.2]{burstall_polynomial_2018} in the smooth case.
It also shows that $\hat{f}$ is independent of the choice of
isothermic sphere congruences (see \cref{th:53}):
\begin{proposition}
  \label{th:36}
  Let $f:\dom\to\cZ$ be an applicable Legendre map and $s^{+}<f$
  an isothermic net.  Let $\hat{s}^+$ be a Darboux transform
  of $s^+$ with parameter $m\in\R^{\times}$ and set
  $\hat{f}:=\hat{s}^+\oplus (f\cap (\hat{s}^{+})^{\perp})$.

  Then $\hat{f}:\dom\to\cZ$ is a Darboux transform of $f$ with
  parameter $m$ and all such arise this way.
\end{proposition}
\begin{proof}
  Choose a Darboux transform $s^-$ of $s^+$ so that
  $f=s^+\oplus s^-$ and let $s=s^+\sqcup s^-$.  Fix some
  $i_0\in\dom$ and let $\hat{s}$ be the Darboux transform of
  $s$ with parameter $m$ which coincides with
  $\hat{s}^+_{i_{0}}$ at $(0,i_0)$.  Then the uniqueness
  assertion in \cref{th:22} tells us that
  $\hat{s}\restr{\set{0}\times \Sigma}=\hat{s}^+$.  Set
  $\hat{s}^-=\hat{s}\restr{\set{1}\times\dom}$.  The Moutard
  equation relating $s^{\pm}_i$ and $\hat{s}^{\pm}_i$ tells
  us that $f_i$ and $(\hat{s}^+\oplus \hat{s}^-)_i$ are
  coplanar and so intersect, necessarily at
  $f\cap (\hat{s}^{+})^{\perp}$.  It follows at once that
  $\hat{f}$ and the Darboux transform
  $\hat{s}^{+}\oplus \hat{s}^-$ coincide.
\end{proof}

\subsubsection{Calapso transformations}
\label{sec:calapso-transf}

\begin{proposition}
  \label{th:38}
  Let $f=s^+\oplus s^-:\dom\to\cZ$ be an applicable Legendre
  map spanned by an isotropic Darboux pair of isothermic
  nets and set $s:s^+\sqcup s^-:\set{0,1}\times\dom\to\QQ$.

  For $t\in\R$, let $s(t)=s^+(t)\sqcup s^-(t)$ be the
  Calapso transform of $s$ and set $f(t):=s^{+}(t)\oplus
  s^-(t)$.

  Then $f(t):\dom\to\cZ$ is an applicable Legendre map that
  we call the \emph{Calapso transform of $f$}.
\end{proposition}
\begin{proof}
  According to \cref{th:23}, the edge label on vertical
  edges of $s(t)$ is $\infty$ so that $s^{\pm}(t)$ are again
  an isotropic Darboux pair and so their span $f(t)$ is
  an applicable Legendre map.
\end{proof}

Again, the construction is independent of the choice of
isothermic nets in $\QQ$.  To see this and to make
contact with the discussion in
\cite[\S3]{burstall_discrete_2018}, we observe that
\cref{th:21}, applied to the vertical quadrilaterals of $s$ yields:
\begin{lemma}[{c.f.\ \cite[Corollary 3.8]{burstall_discrete_2018}}]
  \label{th:37}
  Let $(s^{\pm},\eta^{\pm})$ be an isotropic Darboux pair of
  isothermic nets.  Let $\Gamma^{\pm}(t)$ be the
  corresponding families of flat connections and
  $\tau\in\Gamma s^+\wedge s^-$ such that
  $\eta^{-}=\eta^++\d\tau$.  Then, for all $t\in\R$,
  \begin{equation}
    \label{eq:54}
    (\exp t\tau)\cdot\Gamma^+(t)=\Gamma^-(t).
  \end{equation}
\end{lemma}

With this in hand, we let $T(t)=T^+(t)\sqcup
T^-(t):\set{0,1}\times\dom\to\Ortho(p+1,q+1)$ trivialise
$\Gamma^s(t)$ and note that, by \cref{th:37}, we may take
$T^-(t)=T^+(t)\exp(-t\tau)$.  Thus
\begin{equation*}
  f(t)=T^+(t)s^+\oplus T^-(t)s^{-}=T^{+}(t)(s^+\oplus \exp -t
  \tau s^{-})=T^+(t)(s^+\oplus s^-)=T^+(t)f,
\end{equation*}
since $\exp -t\tau s^-=s^-$.  We conclude that our Calapso
transforms coincide with those of
\cite[Theorem~3.9]{burstall_discrete_2018}:
\begin{proposition}
  \label{th:39}
  Let $f:\dom\to\cZ$ be an applicable Legendre map and
  $s^+<f$ an isothermic net with flat connections
  $\Gamma^+(t)$ trivialised by
  $T^{+}(t):\dom\to\Ortho(p+1,q+1)$.

  Then the Calapso transform $f(t)$ of $f$ is given by
  \begin{equation*}
    f(t)=T^+(t)f.
  \end{equation*}
\end{proposition}

\subsection{Edge-labelling of an applicable Legendre
  map}
\label{sec:fact-funct-an}

The edge labelling of an isothermic sphere congruence $s<f$
is, in fact, an invariant of the applicable Legendre map
$f$ and can be computed from any $\eta\in[\eta]$.
\begin{proposition}
  \label{th:17}
  Let $(f,\eta):\dom\to\cZ=\cZ^{p,q}$ be applicable and, for
  each edge $ij$, choose $\sigma_i\in f_i$, $\sigma_j\in
  f_j$ such that $\eta_{ji}=\sigma_j\wedge\sigma_i$ and set
  $m_{ij}^{\eta}=1/(\sigma_i,\sigma_j)\in\R^{\times}\cup\set{\infty}$
  to get a well-defined function $m^{\eta}$ on edges.

  Then, for any $\tau\in\Gamma\Wedge^2f$,
  \begin{equation*}
    m^{\eta+\d\tau}=m^{\eta}.
  \end{equation*}
\end{proposition}
\begin{proof}
  On an edge $ij$, there are $\lambda_j,\lambda_j\in\R$ and
  $\sigma_{ij}\in s_{ij}$ such that
  \begin{equation*}
    \tau_i=\lambda_i\sigma_i\wedge\sigma_{ij},\qquad
    \tau_j=\lambda_j\sigma_j\wedge\sigma_{ij}.
  \end{equation*}
  Then a simple calculation gives
  \begin{equation*}
    \eta_{ji}+\d\tau_{ji}=(\sigma_j+\lambda_i\sigma_{ij})\wedge(\sigma_i+\lambda_j\sigma_{ij}).
  \end{equation*}
  Now $s_{ij}$ is orthogonal to $f_i+f_j$ so that
  \begin{equation*}
    (\sigma_i+\lambda_j\sigma_{ij},\sigma_j+\lambda_i\sigma_{ij})=(\sigma_i,\sigma_{j}),
  \end{equation*}
  whence the result.
\end{proof}
In particular, all isothermic sphere congruences $s<f$ share
the same edge-labelling.


\section{\texorpdfstring{$\Omega$}{Omega}-nets}
\label{sec:discr-texorpdfstr-su}

\subsection{Lie sphere geometry}
\label{sec:lie-sphere-geometry}

For the rest of the paper, we restrict attention to the case
$(p,q)=(3,1)$ which is the setting for Lie sphere geometry.
Here the Lie quadric $\QQ=\QQ^{3,1}$ parametrises oriented
$2$-spheres in $S^3=\QQ^{3,0}$ and two points of $\QQ$ are
orthogonal exactly when the corresponding $2$-spheres are in
oriented contact.  It follows at once that $\cZ=\cZ^{3,1}$
parametrises oriented contact elements in $S^3$.  All this
requires the choice of a \emph{point sphere complex}, that
is $\p\in\R^{4,2}$ with $(\p,\p)=-1$.  Then
$S^3=\P(\cL\cap\p^{\perp})$ and the $2$-sphere corresponding
to $q\in\QQ$ is $\P(\cL\cap\p^{\perp}\cap q^{\perp})$.

We get a more practical take on these matters via
stereoprojection onto $\R^{3,1}$: thus choose
$\fo,\q\in\cL\cap\p^{\perp}$ with $(\fo,\q)=-1$ and, as
usual, set $\R^{3,1}=\spn{\fo,\q}^{\perp}$.  Then
$\p\in\R^{3,1}$ and we have a further orthogonal
decomposition:
\begin{equation*}
  \R^{3,1}=\R^3\oplus\spn{\p}.
\end{equation*}
Now a point $z=c+r\p\in\R^{3,1}$, with $c\in\R^3$,
parametrises the oriented $2$-sphere in $\R^{3}$ with centre
$c$ and (signed) radius $r$.  This is the intersection of
$\R^3$ with the affine light-cone
$\set{v\in\R^{3,1}\st (v-z,v-z)=0}$ centred at $z$.  Again,
a contact element $(x,n)\in\R^3\times S^2$ of $\R^3$
corresponds to the affine null line through $x$ in the
direction $n+\p$.  This is the Laguerre picture of Lie
sphere geometry: see \cite{cecil_lie_2008} for a more
detailed discussion.

Now contemplate a Legendre map $f:\dom\to\cZ$ and
choose\footnote{This is a countable number of open conditions.}
$\p,\q$ so that:
\begin{compactenum}
\item\label{item:1} $f_i^{\perp}\cap\spn{\p,\q}=\set{0}$,
  for all $i\in\dom$;
\item\label{item:2} The curvature sphere
  $s_{ij}=f_i\cap f_j$ lies in neither $\p^{\perp}$ nor
  $\q^{\perp}$, for all edges $ij$.
\end{compactenum}
In view of condition \ref{item:1}, we have sections $y,\t$
of $f$ with
\begin{align*}
  (y,\q)&=-1&(y,\p)&=0\\(\t,\q)&=0&(\t,\p)&=-1
\end{align*}
Now condition \ref{item:2} tells us that $y_i,\t_{i}$ do not lie
in any $s_{ij}$ so that $y,\t$ both span $Q$-nets which
are regular by \cref{th:3}.

Stereoprojection onto $\R^{3,1}$ now yields
$x=\pi y:\dom\to\R^3$ and $n+\p=\pi\t$ with $n:\dom\to S^2$.
On each edge $ij$, we have a \emph{principal curvature}
$\kappa_{ij}\in\R^{\times}$ with
\begin{equation}
  \label{eq:31}
  \kappa_{ij}x_i+(n_{i}+\p)=\kappa_{ij}x_j+(n_j+\p),
\end{equation}
this common value being the stereoprojection of $s_{ij}$.
Otherwise said, $(x,n)$ comprise a \emph{principal
  (contact element) net}
(c.f.~\cite[Definition~3.24]{bobenko_discrete_2008}).  In
particular, $x$ and $n$ are edge-parallel.

Conversely, given a principal net $(x,n)$, we take $y$ to be
the Euclidean lift of $x$, set $\t=n+\p+(x,n)\q$ and
$f=\spn{y,\t}$ to recover a Legendre map $f:\dom\to\cZ$.
For further discussion, see
\cite[\S3.5]{bobenko_discrete_2008} or
\cite[\S2]{burstall_discrete_2018}.

\subsection{\texorpdfstring{$\Omega$}{Omega}-nets}
\label{sec:discr-texorpdfstr-ne}

With these preparations in hand, we recall from
\cite[Definition~3.1]{burstall_discrete_2018}:
\begin{definition}[$\Omega$-net]
  A Legendre map $f:\dom\to\cZ$ is an $\Omega$-net if is
  spanned by \K\ dual sections.
\end{definition}

In view of \cref{thm:applicability}, we have a number of
alternative characterisations:
\begin{theorem}\label{thm:Omega}
  Let $f:\dom\to\cZ$ be a regular Legendre map.  Then
  the following are equivalent:
  \begin{compactitem}
  \item $f$ is an $\Omega$-net;
  \item $f$ is applicable;
  \item $f$ is spanned by an isotropic Darboux pair of
    isothermic sphere congruences $s^{\pm}<f$.
  \end{compactitem}
\end{theorem}
\begin{remark}
  The characterisation via isothermic sphere congruences is
  a direct discrete analogue of Demoulin's original
  formulation of $\Omega$-surfaces \cite{demoulin_sur_1911-2}.
\end{remark}

Here we add another characterisation to the list which is
also due to Demoulin in the smooth case
\cite{demoulin_sur_1911-1}:

\subsubsection{\texorpdfstring{$\Omega$}{Omega}-nets in
  $\R^3$ via associate nets}

Let $f:\dom\to\cZ$ be an $\Omega$-net with stereoprojection
$L:\dom\to\cZ_{\q}$ and corresponding principal net $(x,n)$
so that $L$ is the affine line congruence through $x$
pointing along $n+\p$.

\cref{th:14} provides a dual $\Omega$-net $\dual{L}$ which
cuts $\R^3$ in a net $\dual{x}:\dom\to\R^3$ so that
$(\dual{x},n)$ is also principal.  We call $\dual{x}$ an
\emph{associate net} of $x$ and seek to characterise it in
purely Euclidean terms.

We begin with a lemma:
\begin{lemma}
  \label{th:16}
  Let $(f,[\eta])$ be applicable with $f=\spn{y,\t}$.  Suppose
  that $(\eta\q,\p)=0$.  Then
  \begin{equation}
    \label{eq:32}
    \eta=\eta\q\curlywedge y + \eta\p\curlywedge\t
  \end{equation}
  with both $\eta\q$ and $\eta\p$ edge-wise parallel to $\d y$.
\end{lemma}
\begin{proof}
We work on an edge $ij$.  It is easy to see that a basis for
$f_{ij}$ is given by $y_{ij}:=\half(y_i+y_j)$,
$\t_{ij}:=\half(\t_i+\t_j)$ and $\d y_{ij}$ so that
\begin{equation*}
  \eta_{ij}=\alpha_{ij}\d y_{ij}\wedge y_{ij}+
  \beta_{ij}\d y_{ij}\wedge\t_{ij}+\gamma_{ij}y_{ij}\wedge\t_{ij}.
\end{equation*}
Now $\gamma_{ij}=(\eta\q,\p)=0$ and then
\begin{equation*}
  \eta_{ij}\q=\alpha_{ij}\d y_{ij}\qquad
  \eta_{ij}\p=\beta_{ij}\d y_{ij},
\end{equation*}
whence the result.
\end{proof}

Now we have $\d\dual{x}=\pi\eta\q$, for some
$\eta\in[\eta]$, so that $(\eta\q,\p)=(\d\dual{x},\p)=0$ and
\cref{th:16} applies. Since $\eta\p$ and so $\pi\eta\p$ is
closed, we have $\dual{n}:\dom\to\R^{3,1}$ with
$\d\dual{n}=\pi\eta\p$.  Moreover,
$(\d\dual{n},\p)=(\eta\p,\p)=0$ so that we may adjust the
constant of integration to ensure that
$\dual{n}:\dom\to\R^3$.  In view of \cref{th:16}, both
$\dual{x}$ and $\dual{n}$ are edge-parallel to $x$.
Finally, the exterior derivative of \eqref{eq:32} yields
\begin{equation*}
  0=-\d\eta=\eta\q \curlywedge \d y+\eta\p\curlywedge \d\t,
\end{equation*}
the $\Wedge^{2}\R^3$-component of which reads
\begin{equation}
  \label{eq:33}
  \d\dual{x}\curlywedge\d x+\d\dual{n}\curlywedge\d n=0.
\end{equation}

Conversely, let $f:\dom\to\cZ$ be Legendre with principal
net $(x,n)$ and suppose there are edge-parallel
$\dual{x},\dual{n}$ satisfying \eqref{eq:33}.  We define
$1$-forms
\begin{equation*}
  \alpha:=\d\dual{x}+(x\wedge\d\dual{x})\q\qquad
  \beta:=\d\dual{n}+(x\wedge\d\dual{n})\q,
\end{equation*}
both of which are edge-wise parallel to $\d y$.  Thus, on
any face $\ijkl$, $\alpha,\beta,\d y,\d\t$ all take values
in $W_{\ijkl}$.  Thus we apply \cref{th:6}, first to
conclude that $\alpha,\beta$ are closed since
$\d\dual{x},\d\dual{n}$ are and then, from \eqref{eq:33},
that
\begin{equation*}
  \alpha\curlywedge\d y+\beta\curlywedge \d\t=0.
\end{equation*}
Thus, setting $\eta=\alpha\curlywedge y+\beta\curlywedge\t$
gives a closed $1$-form with $\eta_{ij}\in \Wedge^{2}f_{ij}$.
Finally, $s_{ij}$ is spanned by $\kappa_{ij}y_{ij}+\t_{ij}$
so that $\eta_{ji}\wedge s_{ij}\neq0$ if and only if
\begin{equation*}
  \alpha_{ji}\wedge y_{ij}\wedge\t_{ij}+\kappa_{ij}\beta_{ji}\wedge
  \t_{ij}\wedge y_{ij}\neq 0,
\end{equation*}
or, equivalently,
\begin{equation}
  \label{eq:37}
  \d\dual{x}_{ji}\neq \kappa_{ij}\d\dual{n}_{ji}.
\end{equation}
In this case, $(f,\eta)$ is an $\Omega$-net.

To summarise, we have the discrete analogue of the
discussion in
\cite[Theorem~5.1]{pember_lie_nodate} (c.f.~\cite[\S2.5]{burstall_polynomial_2018}):
\begin{theorem}\label{thm:dualityR3}
  A principal net $(x,n):\dom\to\R^3\times S^2$ lifts to an
  $\Omega$-net if and only if there exist edge-parallel nets
  $\dual{x}, \dual{n}:\dom\to\R^{3}$ satisfying \eqref{eq:33},
  or, equivalently,
  \begin{equation*}
    A(\dual{x},x) + A(\dual{n},n)=0,
  \end{equation*}
  together with \eqref{eq:37}.
  
  In this case, $(\dual{x},n)$ is also a principal net whose
  Legendre lift is an $\Omega$-net when it is regular.
  
  We call $\dual{n}$ an \emph{associate Gauss map of $x$}.
\end{theorem}

\begin{remark}
  A parallel net $\tilde{x}:=\dual{x} + cn$, for $c\in\R$ constant, to an associate
  net is also an associate net: with $\tilde{n}=\dual{n}-cx$
  we still have
  \begin{equation*}
    A(x,\tilde{x}) + A(\tilde{n},n) =0.
  \end{equation*}
  This amounts to replacing $\dual{L}$ with $\dual{L}-c\p$.

  We therefore obtain a one-parameter family of associate nets with associate Gauss maps.
\end{remark}

We conclude our discussion of duality for $\Omega$-nets by
proving a generalisation of \cref{th:11} and so Bianchi's
formula \cref{th:13} to the present setting.  As we shall
see below in \cref{th:25}, this will also generalise a
discrete version of a formula of Eisenhart for Guichard
surfaces.

From \cref{th:17}, we know that the edge-labelling of an
isothermic sphere congruences $s<f$ is an invariant of the
applicable Legendre map $f$.  The geometry of this
edge-labelling is given by the following generalisation of
\eqref{eq:27}:
\begin{theorem}
  \label{th:18}
  Let $(x,n)$ be a principal net lifting to an $\Omega$-net
  with edge-labelling $m_{ij}$, associate net $\dual{x}$ and
  associate Gauss map $\dual{n}$.  Then, for each edge $ij$,
  \begin{equation}
    \label{eq:34}
    (\d x_{ij},\d\dual{x}_{ij})+(\d n_{ij},\d\dual{n}_{ij})=
    -\frac{2}{m_{ij}}.
  \end{equation}
\end{theorem}
\begin{proof}
  The $\Omega$-net $f=\spn{y,\t}$ has an $\eta$ with
  $\pi\eta\q=\d\dual{x}$, $\pi\eta\p=\d\dual{n}$ and
  $(\eta\q,\p)=0$.  On an edge $ij$, we write
  $\eta_{ji}=\sigma_j\wedge\sigma_i$ where
  \begin{equation*}
    \sigma_i=a_iy_i+b_i\t_i\qquad \sigma_j=a_jy_j+b_j\t_j,
  \end{equation*}
  for constants $a_i,a_j,b_i,b_j$.  Then
  $1/m_{ij}=(\sigma_i,\sigma_j)$ by \cref{th:17}.

  Now $(\eta\q,\p)=0$ yields
  \begin{equation}\label{eq:35}
    a_ib_j=a_jb_i
  \end{equation}
  and then contracting $\eta_{ji}$ with $\q,\p$ and
  projecting tells us that
  \begin{subequations}\label{eq:36}
    \begin{align}
      \d\dual{x}_{ji}&=a_{i}a_j\d x_{ji}+a_ib_j\d n_{ji}\\
      \d\dual{n}_{ji}&=a_jb_i\d x_{ji}+b_ib_j\d n_{ji}.
    \end{align}
  \end{subequations}

  On the other hand, in view of \eqref{eq:35},
  \begin{equation*}
    (\sigma_i,\sigma_j)=a_ia_j(y_i,y_j)+a_{i}b_j\bigl((y_i,\t_j)+(y_j,\t_i)\bigr)+b_ib_j(\t_i,\t_j)
  \end{equation*}
  while
  \begin{align*}
    (y_i,y_j)&=-\half(\d x_{ij},\d x_{ij}),\\
    (y_i,\t_j)+(y_j,\t_i)&=-(\d x_{ij},\d n_{ij}),\\
    (\t_i,\t_j)&=-\half(\d n_{ij},\d n_{ij}),
  \end{align*}
  where the first identity is \eqref{eq:24}, the second is
  proved similarly as is the last using $(n,n)\equiv
  1$.  Thus, using \eqref{eq:36}, we get
  \begin{align*}
    (\sigma_i,\sigma_j)&=-\half\bigl(
                         a_ia_j(\d x_{ij},\d x_{ij})
                         +2a_{i}b_j(\d x_{ij},\d n_{ij})+
                         b_ib_j(\d n_{ij},\d n_{ij})\bigr)\\
    &=-\half\bigl((\d x_{ij},\d\dual{x}_{ij})+(\d n_{ij},\d\dual{n}_{ij})\bigr),
  \end{align*}
  whence the result.
\end{proof}


\section{Guichard nets}

In this section, we define Guichard nets via a direct
analogue of Guichard's original formulation
\cite{guichard_sur_1900}.  We shall see in section
\ref{sec:o-systems} that our definition is
equivalent to that of Schief via discrete $O$-surfaces
\cite{Sch03} and to a formulation through special
$\Omega$-nets of type $1$ that is the discrete version of
the discussion in \cite{burstall_polynomial_2018}, see
Section~\ref{sec:spec-texorpdfstr-net}.

\subsection{Guichard nets via associate surfaces}
\label{sec:guichard-nets-via}

\begin{definition}[Guichard net]\label{th:30}
  A principal net $(x,n):\dom\to\R^3\times S^2$ is
  \emph{Guichard} if and only if there exists an
  edge-parallel net $\dual{x}:\dom\to\R^3$ such that
  \begin{equation}
    \label{eq:39}
    \d\dual{x}\curlywedge\d x+\d n\curlywedge\d n=0,
  \end{equation}
  or, equivalently,
  \begin{equation}
    \label{eq:48}
    A(\dual{x},x)+A(n,n)=0.
  \end{equation}
  In this case, $(\dual{x},n)$ is also Guichard.

  We call $\dual{x}$ an \emph{associate net of $x$}. 
\end{definition}

We see at once that Guichard nets are $\Omega$ (a result of
Demoulin \cite{demoulin_sur_1911} in the smooth case) and
are characterised among $\Omega$-nets by the requirement
that the associate Gauss map $\dual{n}$ may be taken to be
$n$.

In particular, \cref{th:18} immediately yields:
\begin{theorem}
  \label{th:24}
  Let $(x,n)$ be a Guichard net with edge-labelling $m_{ij}$
  and associate net $\dual{x}$.  Then, on each edge $ij$,
  \begin{equation}
    \label{eq:49}
    (\d x_{ij},\d \dual{x}_{ij})+(\d n_{ij},\d n_{ij})=-\frac{2}{m_{ij}}.
  \end{equation}
\end{theorem}

This has a pretty reformulation due to Eisenhart
\cite[p.~210]{Eis14} in the smooth case.  For this, recall
that $r_{ij}:=1/\kappa_{ij}$ is the radius of the $2$-sphere
in oriented contact with both $(x_i,n_i)$ and $(x_j,n_j)$
so that
\begin{align*}
  \d x_{ji}+r_{ij}\d n_{ji}&=0\\
  \intertext{and, similarly,}
  \d\dual{x}_{ji}+\dual{r}_{ij}\d n_{ji}&=0.
\end{align*}
This, along with \eqref{eq:49}, yields:
\begin{corollary}[Eisenhart's formula]
  \label{th:25}
  Let $(x,n)$ be Guichard with associate net $\dual{x}$ and
  edge-labelling $m_{ij}$.  On each edge $ij$, the line
  segments from $x_i$ to $x_j$ and from $\dual{x}_i$ to
  $\dual{x}_j$ are parallel with directed lengths
  $d_{ij},\dual{d}_{ij}$.  Then
  \begin{equation}\label{eq:50}
    d_{ij}\dual{d}_{ij}(1+\frac1{r_{ij}\dual{r}_{ij}})=-\frac{2}{m_{ij}}.
  \end{equation}
\end{corollary}
\begin{remark}
  Eisenhart's original formulation reads:
  \begin{equation*}
    d_{ij}\dual{d}_{ij}+\frac {d_{ij}^2}{r_{ij}^{2}}=-\frac{2}{m_{ij}},
  \end{equation*}
  which is equivalent to our more symmetric version thanks
  to the identity
  $d_{ij}/r_{ij}=\dual{d}_{ij}/\dual{r}_{ij}$.
\end{remark}

\begin{xmpl}\label{th:31}
  According to Bobenko--Pottmann--Wallner \cite{MR2657431},
  a principal net $(x,n)$ in $\R^3$ has mean curvature $H$ and Gauss
  curvature $K$ given by
  \begin{equation*}
    H=-A(x,n)/A(x,x),\qquad K=A(n,n)/A(x,x).
  \end{equation*}
  Thus a principal net is linear Weingarten 
  if there are constants
  $\alpha,\beta,\gamma$, not all zero, with
  \begin{equation*}
    \alpha K +2\beta H +\gamma=0,
  \end{equation*}
  or, equivalently,
  \begin{equation*}
    \alpha A(n,n)-2\beta A(x,n)+\gamma A(x,x)=0.
  \end{equation*}
  Linear Weingarten nets are $\Omega$
  \cite[Theorem~2.8]{burstall_discrete_2018} but more is
  true.  If $\alpha\neq 0$, then $(x,n)$ is Guichard with
  associate net $(\gamma x - 2\beta n)/\alpha$.  In
  particular, when $\beta=0$, that is, $K$ is constant, $x$
  is, up to scale, self-associate just as in the smooth case
  \cite[p.~229]{Eis14}.

  On the other hand, if $\alpha=0$, that is, $H$ is
  constant, $x$ is isothermic with Christoffel dual
  $\gamma x-2\beta n$ which is $n+Hx$ up to scale.  In
  particular, $(x,n)$ is minimal, $H=0$, if and only if
  \begin{equation*}
    \d x\curlywedge\d n=0,
  \end{equation*}
  (c.f. \cite{MR3279541}).
\end{xmpl}

Let $f=\spn{y,\t}:\dom\to\cZ$ be the Legendre lift of a
Guichard net $(x,n)$.  Thus $f$ is applicable with $1$-form
$\eta$ satisfying
\begin{subequations}\label{eq:51}
  \begin{align}
    \pi\eta\q&=\d\dual{x}\\
    \eta\p&=\d\t.
  \end{align}
\end{subequations}

Now let $s^{+}<f$ with $(s^+,\eta^+)$ isothermic and
$\eta^+=\eta+\d\tau^{+}$, for some
$\tau^+\in\Gamma\Wedge^2f$. Set $\sigma^-:=\t+\tau^+\p$.
We have:
\begin{align*}
  \d\sigma^-&=\eta^+\p\\
  (\sigma^-,\p)&=-1.
\end{align*}
Further let $\sigma^+\in\Gamma s^+$ be the affine
lift\footnote{We cheerfully assume that $s^+$ is never
  orthogonal to $\p$.} with
$(\sigma^+,\p)=-1$.  The argument of \cref{th:1} tells us
that $\sigma^-$ is \K\ dual to $\sigma^+$ so that
$\eta^+=\d\sigma^-\curlywedge\sigma^+$. This proves one half
of:
\begin{theorem}
  \label{th:26}
  An applicable Legendre map $f:\dom\to\cZ$ lifts a Guichard net
  if and only if it contains\footnote{In general, $s^{\pm}$
    will span $f$ but it is possible for $s^{\pm}$ to
    coincide: see
    \cite[Lemma~3.11]{burstall_discrete_2014}.} \K\ dual
  sections $\sigma^{\pm}$ with $(\sigma^{\pm},\p)\equiv -1$.
\end{theorem}
\begin{proof}
  Let $f$ lift $(x,n)$.  We need only prove the reverse implication.  For this,
  write $\sigma^-=\t+\tau^+\p$, for some
  $\tau^+\in\Gamma\Wedge^2f$, and consider
  $\eta:=\eta^+-\d\tau^+$.  Then
  \begin{equation*}
    \eta\p=(\d\sigma^-\curlywedge\sigma^+)\p-\d\tau^{+}\p=\d\sigma^--\d\tau^+\p=\d\t.
  \end{equation*}
  In particular, $(\eta\p,\q)=0$ and so we have $\dual{x}:\dom\to\R^3$,
  edge-parallel to $x$, with $\d\dual{x}=\pi\eta\q$.  Thus
  $\dual{x}$ is associate to $x$ with associate Gauss map
  $n$.  Otherwise said, $(x,n)$ is Guichard.
\end{proof}

This analysis yields more.  Let $(x,n)$ be Guichard with
Legendre lift $(f,\eta)$ satisfying \eqref{eq:51}.  We have
seen that $f$ contains isothermic sphere congruences
$s^{\pm}$ with \K\ dual lifts $\sigma^{\pm}$ and that there
is $\tau^{+}\in\Gamma\Wedge^{2}f$ with
$\eta^+=\eta+\d\tau^+$ and $\sigma^-=\t+\tau^+\p$.  Now set
$\tau^-:=\tau^++\sigma^+\wedge\sigma^-$.  An easy
computation shows that the situation is completely
symmetric:
\begin{align*}
  \eta^{\pm}&=\eta+\d\tau^{\pm}\\
  \sigma^{\pm}&=\t+\tau^{\mp}\p,
\end{align*}
where we have used \eqref{eq:15} for the first identity.  In
particular, by \eqref{eq:28}, $x^{\pm}$ have Christoffel
duals $\dual{x}^{\pm}=\dual{x}+\pi\tau^{\pm}\q$.  The
corresponding $2$-spheres in $\R^{3}$ have signed radii
$r^{\pm}=-1/(\sigma^{\pm},\q)$ and
$\dual{r}^{\pm}=-(\dual{x}^{\pm},\p)$.  However,
\begin{equation*}
  (\sigma^{\pm},\q)=(\t+\tau^{\mp}\p,\q)=-(\p,\tau^{\mp}\q)=-(\p,\dual{x}+\tau^{\mp}\q)
  =-(\p,\dual{x}^{\mp}).
\end{equation*}
We therefore conclude, as does Demoulin
\cite{demoulin_sur_1911-1} in the smooth case:
\begin{theorem}
  \label{th:27}
  Let $(x,n)$ be Guichard with associate net $\dual{x}$.
  Then $(x,n)$ and $(\dual{x},n)$ are enveloped by
  Christoffel dual isothermic sphere congruences
  $x^{\pm},\dual{x}^{\pm}$ with radii satisfying
  \begin{equation*}
    r^{\pm}\dual{r}^{\mp}=-1.
  \end{equation*}
\end{theorem}

\subsection{Special \texorpdfstring{$\Omega$}{Omega}-nets}
\label{sec:spec-texorpdfstr-net}
There is another approach to Guichard nets and other
reductions of $\Omega$-nets which arises by applying the
theory of special isothermic nets, due to Bianchi
\cite{bianchi_ricerche_1905} in the classical smooth case
and developed in
\cite{burstall_discrete_2014,burstall_discrete_2015}.  The
virtue of this viewpoint is that it plays well with the
transformation theory.

Recall:
\begin{definition}[Special isothermic net, special $\Omega$-net]
  Let $s:\dom\to\QQ^{p,q}$ be isothermic with family of flat
  connections $\Gamma^s(t)$.  A \emph{polynomial conserved
    quantity of degree $d\in\N$} is a family
  $p(t)=\sum_{n=0}^dp^{(n)}t^n$ of sections of the trivial
  bundle $\dom\times\R^{p+1,q+1}$, with $p^{(d)}$ not
  identically zero, such that $p(t)$ is
  $\Gamma^s(t)$-parallel:
  \begin{equation}\label{eq:58}
    \Gamma^s(t)_{ji}p(t)_i=p(t)_j.
  \end{equation}

  If $s$ admits such a polynomial conserved quantity, we say
  that $s$ is a \emph{special isothermic net of type $d$}.

  If $f$ is an $\Omega$-net, we say that $f$ is a
  \emph{special $\Omega$-net of type $d$} if it is spanned by
  isothermic sphere congruences $s^{\pm}$ for which $s=s^+\sqcup
  s^-$ is special isothermic of type $d$.
\end{definition}

The condition that $f$ be special $\Omega$ is independent of
the choice of isothermic sphere congruence:
\begin{lemma}
  \label{th:40}
  $f:\dom\to\cZ$ is special $\Omega$ of type $d$ if and
  only if there is $s^+<f$ which has a degree $d$ polynomial conserved
  quantity $p^+(t)$ with $(p^+)^{(d)}\perp f$.
\end{lemma}
\begin{proof}
  Let $f=s^+\oplus s^-$ be special $\Omega$ so that
  $s=s^+\sqcup s^-$ has a degree $d$ conserved quantity
  $p(t)=p^{+}(t)\sqcup p^-(t)$.  If $\eta^{-}=\eta^++\d\tau$
  then $p(t)$ being $\Gamma^s(t)$-parallel on vertical edges
  amounts to:
  \begin{equation}
    \label{eq:60}
    \exp t\tau p^{+}(t)=p^-(t).
  \end{equation}
  Certainly $p^+(t)$ is a polynomial conserved quantity for
  $s^{+}$.  Moreover, the left side of \eqref{eq:60} has a term of degree
  $d+1$ with coefficient $\tau(p^{+})^{(d)}$ which must
  vanish since $\deg p^{-}(t)=d$.  This is equivalent to the
  demand that $(p^{+})^{(d)}\perp f$.

  For the converse, given $s^+,p^{+}(t)$, choose an
  $m=\infty$ Darboux transform $s^-$ so that $f=s^{+}\oplus
  s^-$.  With $\tau\in \Gamma s^+\wedge s^-$ such that
  $\eta^-=\eta^++\d\tau$, define $p^-(t)$ by \eqref{eq:60}
  which is of degree $d$ since $(p^{+})^{(d)}\perp f$.  Now
  set $p(t)=p^+(t)\sqcup p^-(t)$ to get a degree $d$
  conserved quantity of $s^+\sqcup s^-$.
\end{proof}

In this context, there is a useful reformulation when $d=1$:
\begin{lemma}
  \label{th:32}
  Let $(s^+,\eta^+):\dom\to\QQ$ be isothermic with flat
  connections $\Gamma^+(t)$ and edge-labelling
  $m_{ij}\neq\infty$.  Then $p^+(t)=\c+t\xi$ is a
  polynomial conserved quantity if and only if
  \begin{subequations}
    \label{eq:53}
    \begin{align}
      \d \c&=0\label{eq:55}\\
      \d \xi&=\eta^+ \c\label{eq:56}\\
      \xi&\perp s^+.\label{eq:57}
    \end{align}
  \end{subequations}
\end{lemma}
\begin{proof}
  If $p^+(t)$ is $\Gamma^+(t)$-parallel then evaluating
  \eqref{eq:58} at $t=0$ yields \eqref{eq:55} while
  differentiating it at $t=0$ gives \eqref{eq:56}.  Finally
  $\Gamma^+(t)_{ji}p^+(t)_{i}$ has a quadratic term along
  the $s^+_j$-component of $\xi_i$ and a simple pole along
  the $s^+_{i}$-component of $p^+(m_{ij})_i$.  These terms
  must both vanish, or, equivalently,
  \begin{equation}
    \label{eq:59}
    \xi_{i}\perp s^+_i\qquad p^+(m_{ij})_{i}\perp s^+_j.
  \end{equation}
  However, with $\mu\in\Gamma s^+$ the Moutard lift, we use
  $\eta^+_{ji}=\mu_j\wedge\mu_i$ and
  $m_{ij}=1/(\mu_i,\mu_j)$ along with \eqref{eq:56} to get
  \begin{equation*}
    (\xi_j,\mu_j)=\frac{1}{m_{ij}}(p^+(m_{ij})_i,\mu_j).
  \end{equation*}
  Thus, given \eqref{eq:56}, \eqref{eq:57} is equivalent to
  \eqref{eq:59}.

  For the converse, if \eqref{eq:53} holds then
  $\Gamma^+(t)_{ji}p^+(t)_i-p^+(t)_{j}$ is a degree $1$
  polynomial in $t$ with vanishing $1$-jet at $t=0$ and so
  vanishes identically.
\end{proof}

Since the $\Gamma(t)$ are metric connections, when $p(t)$ is
a degree $d$ conserved quantity, $(p(t),p(t))$ is a constant
coefficient polynomial in $t$ of degree no more than $2d$
which encodes much of the geometry of the situation.  For
example, let us take $d=1$ and suppose that $(p(t),p(t))$ is
affine linear.  We prove:
\begin{theorem}
  \label{th:33}
  Let $f:\dom\to\cZ$ be an $\Omega$-net.  Then $f$ lifts a
  Guichard net if and only if it is special $\Omega$ with
  linear conserved quantity $p(t)$ satisfying
  $(p(t),p(t))=-1-2ct$, for $c$ a non-zero constant.
\end{theorem}
\begin{proof}
  If $f$ lifts a Guichard net, let $\sigma^{\pm}$ be the \K\
  dual sections provided by \cref{th:26} and set
  $s^+=\spn{\sigma^+}$.  Then
  $\eta^+=\d\sigma^-\curlywedge\sigma^+$ so that
  $\d\sigma^-=\eta^+\p$.  Now $p^+(t):=\p+t\sigma^-$ is a
  polynomial conserved quantity for $(s^+,\eta^+)$ by
  \cref{th:32} and $(p^{+}(t),p^{+}(t))=-1-2t$.  Since
  $\sigma^-\perp f$, \cref{th:40} tells us that $f$ is
  special $\Omega$ with a conserved quantity of the desired
  kind.
 
  For the converse, let $s=s^+\sqcup s^-$ have linear
  conserved quantity $p(t)=p^+(t)\sqcup p^-(t)$ with
  $(p(t),p(t))=-1-2ct$.  By scaling $t$ (and so $\eta$,
  $m_{ij}$, see \cref{th:43}), if necessary, we may take
  $c=1$ and, without loss of generality, set the constant
  $p^{(0)}=\p$.  Write $p^+(t)=\p+t\xi$ and note that
  $\xi\in\Gamma f$ being a null section of $f^{\perp}$.
  Moreover, let $\sigma^+\in\Gamma s^+$ with
  $(\sigma^+,\p)=-1$.  Since $(\xi,\p)=-1$ also,
  \eqref{eq:56} tells us that $\xi$ is \K\ dual to
  $\sigma^+$.  Thus $f$ lifts a Guichard net by
  \cref{th:26}.
\end{proof}

\subsubsection{Variations on the theme}
\label{sec:variations-theme}

We have just seen that Guichard surfaces are special
$\Omega$ nets of type $1$ with
affine linear $(p(t),p(t))$.  Other specialisations of
$\Omega$-nets are obtained by varying
the possibilities for the linear function
$(p(t),p(t))=a+bt$.  We may argue as in \cref{th:33} to
conclude that such $\Omega$-nets contain \K\
dual lifts $\sigma^{\pm}$ with
\begin{equation*}
  (\sigma^+,p^{(0)})=-1,\qquad (\sigma^-,p^{(0)})=b/2.
\end{equation*}
Moreover, we can scale $p(t)$ by a constant
to ensure that $a^2=1$ or $0$ and then a constant rescaling
of $t$ allows us to assume that $b=-2$ or $0$.  This gives
four possibilities:
\begin{compactenum}
\item $a^{2}=1$ and $b=-2$.  For $a=-1$, this is the
  Guichard case we have already studied while $a=1$ gives
  the entirely analogous theory of Guichard nets in
  $\R^{2,1}$.
\item $a^2=1$ and $b=0$.  Here we have
  \begin{equation*}
    (\sigma^+,\p)=-1,\qquad (\sigma^{-},\p)=0
  \end{equation*}
  where $\p=p^{(0)}$ with $(\p,\p)=a$.

  Thus $\sigma^{-}$ spans an isothermic net with
  stereoprojection $x$ in either $\R^{3}$ or $\R^{2,1}$.  We
  therefore have a principal net $(x,n)$ with $x$ isothermic
  or equivalently, $\dual{n}$ is constant.  Conversely, any
  such principal net arises this way.
\item $a^2=0$ and $b=-2$.  Here we may take $p^{(0)}=\q$ and then
  \begin{equation*}
    (\sigma^{\pm},\q)=-1.
  \end{equation*}
  Equivalently, the corresponding principal net $(x,n)$ has
  $\dual{x}=x$ so that
  \begin{equation*}
    \d x\curlywedge\d x+\d\dual{n}\curlywedge\d n=0.
  \end{equation*}
  We call such nets \emph{$L$-Guichard} and remark that
  these nets are additionally characterised by having
  self-dual Legendre maps: $L=\dual{L}$.
\item $a^2=b=0$.  Here we have
  \begin{equation*}
    (\sigma^+,\q)=-1,\qquad(\sigma^-,\q)=0
  \end{equation*}
  so that the tangent sphere congruence $\spn{\t}$ is
  isothermic.  Equivalently, $\dual{x}$ is constant so that
  \begin{equation*}
    \d n\curlywedge\d\dual{n}=0.
  \end{equation*}
  With Bobenko--Suris \cite[\S5]{bobenko_isothermic_2007} we
  call such nets \emph{$L$-isothermic}\footnote{For
    Bobenko--Suris, the $L$-isothermic net is the conical
    net in the space of affine $2$-planes in $\R^3$ given
    by the planes through $x$ normal to $n$.} and remark
  that they are characterised by the requirement that
  $(\dual{n},n)$ is a minimal principal net, c.f.\
  \cref{th:31}, or equivalently, that $n$ is isothermic in
  $S^2$ with Christoffel dual $\dual{n}$, c.f.\
  \cite[Theorem~5.3]{bobenko_isothermic_2007}.  See
  \cite{Eis08,MusNic00} for the smooth theory of
  $L$-isothermic surfaces.
\end{compactenum}

\subsection{Transformations}

The transformation theory of special $\Omega$-nets
$f=s^+\oplus s^-$ now proceeds by applying the results of
\cite{burstall_discrete_2014,burstall_discrete_2015,burstall_polynomial_2018}
to the special isothermic net $s=s^+\sqcup s^-$.

\subsubsection{Darboux transformations}

According to \cref{th:35}, a Darboux transformation
$\hat{f}$ of $f$ amounts to a Darboux
transformation $\hat{s}$ of $s$ with parameter $m$.  When
$s$ has a polynomial quantity $p(t)$ of degree $d$ then $\hat{p}(t)$ given
by
\begin{equation*}
  \hat{p}(t):=\Gamma_s^{\hat{s}}(1-t/m)p(t)
\end{equation*}
is also a polynomial conserved quantity of degree $d$ so
long as $\hat{s}\perp p(m)$ \cite[Lemma
4]{burstall_discrete_2015}.  Since $\hat{s}$ and $p(m)$ are
both $\Gamma^{s}(m)$-parallel, this last holds as soon as it
holds at a single point.  In this case, $p(t)$ and
$\hat{p}(t)$ have the same constant term and the constant
polynomials $(p(t),p(t))$ and $(\hat{p}(t),\hat{p}(t))$
coincide.  Thus, so long as $p(m)\perp \hat{s}$, $\hat{f}$
lifts a Guichard, $L$-Guichard, isothermic or $L$-isothermic
net if $f$ does.  In this way, we obtain discrete versions
of the Eisenhart transformations of Guichard surfaces
\cite{Eis14} and the Bianchi--Darboux transformations of
$L$-isothermic surfaces \cite{B1915,Eis08,MusNic00}.

\subsubsection{Calapso transformations}
Again, Calapso transforms $f(t)$ of $f$ amount to Calapso
transforms $s(t)$ of $s$ thanks to \cref{th:38}.  If $s$
has polynomial conserved quantity $p(u)$ then it follows
from \cref{th:23} that $p(t)(u):=T(t)p(u+t)$ is a polynomial
conserved quantity of the same degree \cite[Theorem
3]{burstall_discrete_2015}.  Here, however, the constant
coefficient polynomials differ by a translation:
$(p(t)(u),p(t)(u))=(p(u+t),p(u+t))$.  Thus, for example, we
have:
\begin{proposition}
  \label{th:41}
  Let $f$ lift a Guichard net in $\R^3$ with linear
  conserved quantity $p(t)$.  Then $f(t)$ lifts a Guichard
  net in $\R^3$ or $\R^{2,1}$ or an $L$-Guichard net
  according to whether $(p(t),p(t))$ is negative, positive
  or zero, respectively. 
\end{proposition}


\section{\texorpdfstring{$O$}{O}-systems and \texorpdfstring{$\Omega$}{Omega}-nets}
\label{sec:o-systems-omega}

The theory of $O$-surfaces developed by
Konopelchenko--Schief \cite{SchKon03} and discretised by
Schief \cite{Sch03} offers a uniform approach to integrable surface
geometry in $\R^3$ that fits well with our approach.  We
give a brief account of this theory in our framework and
prove that $\Omega$-surfaces are indeed $O$-surfaces.

\subsection{Combescure transformations}
\label{sec:comb-transf}

\begin{definition}[Combescure transformation]
  Let $x:\dom\to\R^{p,q}$.  A \emph{Combescure
    transformation} of $x$ is a map $x^{*}:\dom\to\R^{p,q}$
  such that
  \begin{compactenum}
  \item $x^{*}$ is edge-parallel to $x$.
  \item $(\d x\wedge \d x^{*})=0$, where $(\,,\,)$ is the
    inner product on $\R^{p,q}$.
  \end{compactenum}
  In this case, we say that $x,x^{*}$ are a \emph{Combescure
  pair}.
\end{definition}

Generically, $x,x^{*}$ are a Combescure pair exactly when they are
edge-parallel circular nets:
\begin{lemma}
  \label{th:28}
  Let $x,x^{*}:\dom\to\R^{p,q}$ be edge-parallel nets with
  non-collinear quadrilaterals.  Define
  $\lambda_{ij}\in\R^{\times}$ by
  $\d x^{*}_{ij}=\lambda_{ij}\d x_{ij}$ on each edge and let
  $\ijkl$ be a quadrilateral on which
  $\lambda_{ij},\lambda_{jk},\lambda_{k\ell},\lambda_{\ell
    i}$ do not all coincide.

  Then the following are equivalent:
  \begin{compactenum}
  \item $x$ is circular on $\ijkl$.
  \item $x^{*}$ is circular on $\ijkl$.
  \item $(\d x\wedge \d x^{*})_{\ijkl}=0$.
  \end{compactenum}
\end{lemma}
\begin{proof}
  Begin by observing that $\d^2x^{*}_{\ijkl}=0$ reads
  \begin{equation*}
    0 = \lambda_{ij}\d x_{ij}+\lambda_{jk}\d
    x_{jk}+\lambda_{k\ell}\d x_{k\ell}+\lambda_{\ell i}\d
    x_{\ell i}=
    \sum_{i}a_{i}x_{i},
  \end{equation*}
  for $a_i=\lambda_{ij}-\lambda_{\ell i}$ and so on.  Since
  $\sum_ia_i=0$, we deduce that
  \begin{equation}\label{eq:52}
    \sum_ia_{i}(x_i+\fo)=0
  \end{equation}
  with not all $a_i$ zero.  Moreover, since our
  quadrilateral is non-collinear, the $a_i$ are uniquely
  determined up to scale by \eqref{eq:52}.

  Now $x$ is circular if and only if the Euclidean lifts
  also satisfy $\sum_{i}a_iy_i=0$ or, equivalently, we have
  $\sum_ia_i(x_i,x_i)=0$ in addition.

  However,
  \begin{align*}
    \sum_ia_i(x_i,x_i)&=\lambda_{ij}\d (x,x)_{ij}+\lambda_{jk}\d
                        (x,x)_{jk}+\lambda_{k\ell}\d (x,x)_{k\ell}+\lambda_{\ell i}\d
                        (x,x)_{\ell i}\\
                      &=2\bigl(\lambda_{ij}(x\wedge\d x)_{ij}+
                        \lambda_{jk}(x\wedge\d x)_{jk}
                        +\lambda_{k\ell}(x\wedge\d x)_{k\ell}+\lambda_{\ell i}
                        (x\wedge\d x)_{\ell i}\bigr)\\
                      &=2\d(x\wedge\d x^{*})_{ijk\ell}=2(\d x\wedge\d x^{*})_{ijk\ell}.
  \end{align*}
  This settles the equivalence of items 1 and 3 and the
  proposition follows by symmetry since
  $(\d x^{*}\wedge\d x)=-(\d x\wedge \d x^{*})$.
\end{proof}
On the other hand, if
$\lambda_{ij}=\lambda_{jk}=\lambda_{k\ell}=\lambda_{\ell
  i}=\lambda$, then
\begin{equation*}
  (\d x\wedge\d x^{*})_{\ijkl}=\lambda(\d
  x\wedge\d x)_{\ijkl}=0
\end{equation*}
automatically so that we have:
\begin{proposition}
  \label{th:29} Let $x,x^{*}:\dom\to\R^{p,q}$ be
edge-parallel nets with non-collinear quadrilaterals.  If
$x$ is circular then $x,x^{*}$ are a Combescure pair.

Conversely, if $x,x^{*}$ are a Combescure pair and there
is no quadrilateral on which
$\lambda_{ij}=\lambda_{jk}=\lambda_{k\ell}=\lambda_{\ell
i}$, then $x,x^{*}$ are both circular.
\end{proposition}

\subsection{\texorpdfstring{$O$}{O}-systems}
\label{sec:o-systems}

Let $x^{\alpha}:\dom\to\R^{p,q}$, $\alpha=1,\dots,N$ be a
family of mutually edge-parallel nets.  Being edge-parallel
is a linear condition so we may view the $x^{\alpha}$ as a
single linear map $\Phi=\sum_{\alpha}x^{\alpha}\tens
w_{\alpha}:\dom\to\R^{p,q}\tens W$ where $W$ is
some\footnote{Intrinsically, $W$ is the dual of the
finite-dimensional subspace of
$\operatorname{Map}(\dom,\R^{p,q})$ spanned by the
$x^{\alpha}$ and $\Phi$ is evaluation.} $N$-dimensional
vector space with basis $w_1,\dots,w_N$.  We recover the
$x^{\alpha}$ from $\Phi$ by contracting $\Phi$ against the
dual basis $w^1,\dots,w^N$.  That the $x^{\alpha}$ are
edge-parallel is equivalent to the demand that each
$\d\Phi_{ji}$ be decomposable:
\begin{equation*}
  \d\Phi_{ji}=X_{ji}\tens Y_{ji},
\end{equation*}
for $X_{ji}\in\R^{p,q}$ and $Y_{ji}\in W$.

Now choose a basis $e_1,\dots,e_{p+q}$ of $\R^{p,q}$ and
write $\Phi=\sum_{m}e_{m}\tens y^{m}$ to define
$y^{m}:\dom\to W$ and observe that the $y^{m}$ are
also edge-parallel: each $\d y^{m}_{ji}\prl Y_{ji}$.  We
call the $y^{m}$ the \emph{dual edge-parallel family}.

With this in hand, we make the following:
\begin{definition}[$O$-system]
  A family $x^{\alpha}:\dom\to\R^{p,q}$ of mutually
  Combescure nets is an \emph{$O$-system} if the dual family
  $y^m$ are also mutually Combescure with respect to a
  non-degenerate inner product on $W$.
\end{definition}
\begin{remark}
  In view of \cref{th:29}, when $\Sigma=\Z^2$, an $O$-system
  generically comprises a family
  of $O$-surfaces in the sense of Schief
  \cite[Definition~5.1]{Sch03}.
\end{remark}

We now give two characterisations of $O$-systems:
\begin{theorem}
  \label{th:7}
  Let $x^{\alpha}:\dom\to\R^{p,q}$ be a family of mutually
  Combescure nets with $\Phi=\sum_{\alpha}x^{\alpha}\tens
  w_{\alpha}:\dom\to \R^{p,q}\tens W$ as above.

  Equip $W$ with a non-degenerate inner product $g$
  and set $g_{\alpha\beta}:=g(w_{\alpha},w_{\beta})$.

  Then the following are equivalent:
  \begin{compactenum}
  \item $x^{\alpha}$ is an $O$-system with respect to
    $g$.
  \item $\sum_{\alpha,\beta}g_{\alpha\beta}\d
    x^{\alpha}\curlywedge\d x^{\beta}=0$.
  \item $[\d\Phi\wedge\d\Phi]=0$, where we identify
    $\R^{p,q}\tens W$ with $\R^{p,q}\wedge
    W\leq\Wedge^2(\R^{p,q}\oplus W)\cong\fo(\R^{p,q}\oplus
    W)$ via \eqref{eq:14}.
  \end{compactenum}
\end{theorem}
\begin{proof}
  For $a\tens v,b\tens w\in\R^{p,q}\tens W$ we have
  \begin{equation*}
    [a\tens v,b\tens w]=(a,b)v\wedge w+g(v,w)a\wedge b\in\Wedge^2W\oplus\Wedge^{2}\R^{p,q}.
  \end{equation*}
  Thus, since $\Phi=\sum_{\alpha}x^{\alpha}\tens
  w_{\alpha}$,
  \begin{equation*}
    [\d\Phi\wedge\d\Phi]=\sum_{\alpha,\beta}\bigl(
    (\d x^{\alpha}\wedge\d x^{\beta})w_{\alpha}\wedge
    w_{\beta}+
    g_{\alpha\beta}\d x^{\alpha}\curlywedge\d x^{\beta}\bigr).
  \end{equation*}
  Since the $x^{\alpha}$ are mutually Combescure, the
  $\Wedge^2W$-component vanishes and we have established the
  equivalence of items 2 and 3.  On the other hand, writing
  $\Phi=\sum_me_m\tens y^m$, the same argument tells us that
  the $\Wedge^{2}\R^{p,q}$-component of
  $[\d\Phi\wedge\d\Phi]$ can also be written as
  \begin{equation*}
    \sum_{m,n}g(\d y^{m}\wedge \d y^n)e_m\wedge e_n,
  \end{equation*}
  which vanishes exactly when the $y^m$ are mutually
  Combescure.  Thus items 1 and 3 are equivalent.
\end{proof}

As a consequence, all the nets we have considered in this
paper come from $O$-systems:
\begin{xmpls}
\item[]
  \begin{compactenum}
  \item Recall from \cref{th:31} that a principal net
    $(x,n)$ is linear Weingarten if there are constants
    $\alpha,\beta,\gamma$, not all zero, such that
    \begin{equation*}
      \alpha\d n\curlywedge\d n -2\beta \d n\curlywedge \d x
      + \gamma\d x\curlywedge\d x=0.
    \end{equation*}
    \cref{th:7} now tells us that this happens exactly when
    $x,n$ comprise an $O$-system with
    \begin{equation*}
      (g_{\alpha\beta})=
      \begin{pmatrix}
        \gamma&-\beta\\-\beta&\alpha
      \end{pmatrix}.
    \end{equation*}
    In particular, our notion of linear Weingarten net
    coincides with that of Schief \cite[\S5(b)(iv)]{Sch03}.
  \item (c.f.~\cite[\S5(b)(ii)]{Sch03}) An isothermic net
    $x$ with Christoffel dual $\dual{x}$ comprise an
    $O$-system with $W=\R^{1,1}$ and
    \begin{equation*}
      (g_{\alpha\beta })=
      \begin{pmatrix}
        0&1\\1&0
      \end{pmatrix},
    \end{equation*}
    thanks to \cref{th:8}.

    The same is true for $n,\dual{n}$ when $x$ is
    $L$-isothermic as in Section~\ref{sec:variations-theme}.
    We conclude that $x$ is $L$-isothermic if and only if it
    is a Combescure transform of a minimal net
    $(\dual{n},n)$ as was observed in the smooth case by
    Schief--Szereszewski--Rogers \cite[\S6]{Schief2009-ly}.
  \item (c.f.~\cite[\S5(b)(v)]{Sch03}) Let $(x,n)$ be a
    Guichard net with associate net $\dual{x}$.  Then
    $x,\dual{x},n$ comprise an $O$-system with $W=\R^{2,1}$
    and
    \begin{equation*}
      (g_{\alpha\beta })=
      \begin{pmatrix}
        0&\half&0\\\half&0&0\\0&0&1
      \end{pmatrix},
    \end{equation*}
    thanks to \cref{th:30}.

    The same is true for $n,\dual{n},x$ when
    $(x,n)$ is $L$-Guichard with associate Gauss map
    $\dual n$.
  \item Finally, let $(x,n)$ be the stereoprojection of an
    $\Omega$-net with associate net $\dual{x}$ and associate
    Gauss map $\dual{n}$.  Then $x,\dual{x},n,\dual{n}$
    comprise an $O$-system with $W=\R^{2,2}$ and
    \begin{equation*}
      (g_{\alpha\beta })=
      \begin{pmatrix}
        0&1&0&0\\1&0&0&0\\0&0&0&1\\0&0&1&0
      \end{pmatrix},
    \end{equation*}
    by \cref{thm:dualityR3}.
  \end{compactenum}
\end{xmpls}

\begin{remark}
  The equation $[\d\Phi\wedge\d\Phi]=0$ is a straightforward
  discretisation of the \emph{curved flat}
  system \cite{FerPed96a} in a version studied by Burstall
  \cite[\S3.1]{burstall_isothermic_2006} (as $\p$-flat maps)
  and Br\"uck--Du--Park--Terng
  \cite[Definition~6.6]{BruDuParTer02} (as $n$-tuples in
  $\R^m$ of type $O(n)$).  We may return to discrete curved
  flats elsewhere.
\end{remark}


\end{document}